\documentclass{siamart1116}
\usepackage{amsfonts}
\usepackage{accents}
\usepackage[margin=1in, letterpaper]{geometry}
\usepackage{graphicx}
\usepackage{bmpsize}
\usepackage{subcaption}
\usepackage{float}
\usepackage{epsfig}
\usepackage{multicol}
\usepackage{cite}
\usepackage{adjustbox,lipsum}

\newcommand{\vertiii}[1]{{\left\vert\kern-0.25ex\left\vert\kern-0.25ex\left\vert #1\right\vert\kern-0.25ex\right\vert\kern-0.25ex\right\vert}}
\begin{document}

\title{A Second Order Ensemble Timestepping Algorithm for Natural Convection}
\author{J. A. Fiordilino\thanks{The author is supported by the DoD SMART Scholarship.  The research herein was also partially supported by NSF grants CBET 1609120 and DMS 1522267.}}
\date{Updated: 6/16/17}
\maketitle

\begin{abstract}
	This paper presents an algorithm for calculating an ensemble of solutions to natural convection problems.  The ensemble average is the most likely temperature distribution and its variance gives an estimate of prediction reliability.  Solutions are calculated by solving two coupled linear systems, each involving a shared coefficient matrix, for multiple right-hand sides at each timestep.  Storage requirements and computational costs to solve the system are thereby reduced.  Moreover, this paper addresses a need for higher order methods to solve natural convection problems.  Stability and convergence of the method are proven under a timestep condition involving fluctuations of the velocity.  Numerical tests are provided which confirm the theoretical analyses.
\end{abstract}
\section{Introduction}
Ensemble calculations are essential in predictions of the most likely outcome of systems with uncertain data; for instance, weather forecasting \cite{Kalnay} and ocean modeling \cite{Lermusiaux}.  Furthermore, they are finding application in an increasing number of fields, including turbulence \cite{Nan2}, magnetohydrodynamics \cite{Moheb}, and 3D printing \cite{Sakthivel}.  Ensemble simulations classically involve J sequential, fine mesh runs or J parallel, coarse mesh runs of a given code.  This leads to a competition between ensemble size and mesh density.  We develop a linearly implicit timestepping method with shared coefficient matrices to address this issue.  For such methods, it is more efficient in both storage and solution time to solve J linear systems with a shared coefficient matrix than with J different matrices.  Prediction of thermal profiles is essential in many applications \cite{Bairi,Gebhart,Marshall,Ostrach}.  Herein, we extend an earlier study \cite{Fiordilino} regarding first order timestepping algorithms for natural convection based on the pioneering work for isothermal flows of Jiang and Layton \cite{Nan}.

Consider natural convection within an enclosed cavity with zero wall thickness, see Figure 1 for a typical setup.  Let $\Omega \subset \mathbb{R}^d$ (d=2,3) be a polyhedral domain with boundary $\partial \Omega$.  The boundary is partitioned such that $\partial \Omega = \overline{\Gamma_{1}}  \cup \overline{\Gamma_{2}}$ with $\Gamma_{1} \cap \Gamma_{2} =\emptyset$, $|\Gamma_{1}| > 0$, and $\Gamma_{1} = \Gamma_{H} \cup \Gamma_{N}$.  Given $u(x,0;\omega_{j}) = u^{0}(x;\omega_{j})$ and $T(x,0;\omega_{j}) = T^{0}(x;\omega_{j})$ for $j = 1,2, ... , J$, let $u(x,t;\omega_{j}):\Omega \times (0,t^{\ast}] \rightarrow \mathbb{R}^{d}$, $p(x,t;\omega_{j}):\Omega \times (0,t^{\ast}] \rightarrow \mathbb{R}$, and $T(x,t;\omega_{j}):\Omega \times (0,t^{\ast}] \rightarrow \mathbb{R}$ satisfy
\begin{align}
u_{t} + u \cdot \nabla u -Pr \Delta u + \nabla p &= PrRa\xi T + f \; \; in \; \Omega, \label{s2} \\
\nabla \cdot u &= 0 \; \; in \; \Omega,  \\
T_{t} + u \cdot \nabla T - \Delta T &= \gamma \; \; in \; \Omega, \label{s2T} \\
u = 0 \; \; on \; \partial \Omega,  \; \; \;
T = 1 \; \; on \; \Gamma_{N}, \; \; \;
T = 0 \; \; on \; \Gamma_{H}, \; \; \;
n \cdot \nabla T &= 0 \; \; on \; \Gamma_{2}, \label{s2f}
\end{align}
\noindent Here $n$ denotes the usual outward normal, $\xi$ denotes the unit vector in the direction of gravity, $Pr$ is the Prandtl number, and $Ra$ is the Rayleigh number.  Further, $f$ and $\gamma$ are the body force and heat source, respectively. 

Let $<u>^{n}_{e} := \frac{1}{J} \sum_{j=1}^{J} (2u^{n} - u^{n-1})$ and ${u'}^{n} = 2u^{n} - u^{n-1} - <u>^{n}_{e}$ be the extrapolated ensemble average and fluctuation; the ensemble average is denoted $<\cdot>$.  To present the idea, suppress the spatial discretization for the moment.  We apply an implicit-explicit (IMEX) time-discretization to the system (\ref{s2}) - (\ref{s2f}), while keeping the coefficient matrix independent of the ensemble members.  This leads to the following timestepping method:
\begin{align}
\frac{3u^{n+1} - 4u^{n} + u^{n-1}}{2\Delta t} + <u>^{n}_{e} \cdot \nabla u^{n+1} + {u'}^{n} \cdot \nabla (2u^{n} - u^{n-1}) - Pr \Delta u^{n+1} + \nabla p^{n+1} \label{d1} \\
= PrRa\xi (2T^{n}-T^{n-1}) + f^{n+1}, \notag \\
\nabla \cdot u^{n+1} = 0, \\
\frac{3T^{n+1} - 4T^{n} + T^{n-1}}{2\Delta t} + <u>^{n}_{e} \cdot \nabla T^{n+1} + {u'}^{n} \cdot \nabla (2T^{n} - T^{n-1}) - \Delta T^{n+1} = \gamma^{n+1}. \label{d2}
\end{align}
By lagging $u'$ and using linear extrapolation for the coupling term $\xi T$ in the method, the fluid and thermal problems uncouple and each sub-problem contains a shared coefficient matrix for all ensemble members.
\begin{figure}
	\centering
	\includegraphics[width=2.5in,height=2.5in, keepaspectratio]{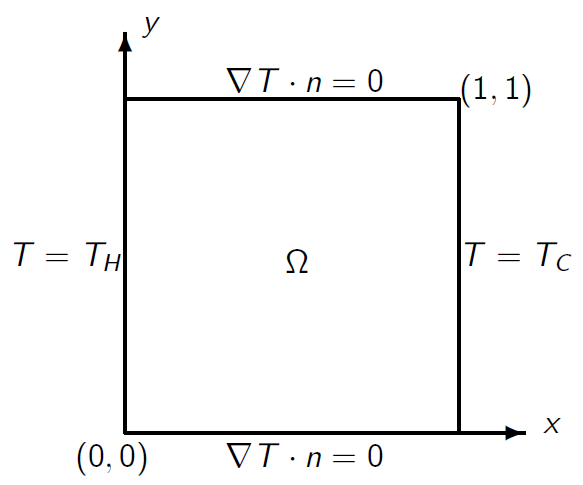}
	\caption{Domain and boundary conditions for double pane window problem benchmark.}
\end{figure}
In Section 2, we collect necessary mathematical tools.  In Section 3, we present an algorithm based on (\ref{d1}) - (\ref{d2}) in the context of the finite element method.  Stability and error analysis of the algorithm follow in Section 4.  In particular, under a CFL-type condition, we prove the stability of the proposed algorithm in Theorem \ref{t1} and its convergence in Theorem \ref{error:thick}.   We end with numerical experiments and conclusions in Sections 5 and 6.
\section{Mathematical Preliminaries}
The $L^{2} (\Omega)$ inner product is $(\cdot , \cdot)$ and the induced norm is $\| \cdot \|$.  Define the Hilbert spaces,
\begin{align*}
X &:= H^{1}_{0}(\Omega)^{d} = \{ v \in H^{1}(\Omega)^d : v = 0 \; on \; \partial \Omega \}, \;
Q := L^{2}_{0}(\Omega) = \{ q \in L^{2}(\Omega) : (1,q) = 0 \}, \\
W &:= H^{1}(\Omega), \; W_{\Gamma_{1}}:= \{ S \in W : S = 0 \; on \; \Gamma_{1} \}, \;
V := \{ v \in X : (q,\nabla \cdot v) = 0 \; \forall \; q \in Q \},
\end{align*}
and $H^{1}(\Omega)$ norm $\| \cdot \|_{1}$.
Moreover, we utilize the fractional order Hilbert space on the non-homogeneous Dirichlet boundary $H^{1/2}(\Gamma_{N})$ with corresponding norm
$$\| R \|_{1/2,\Gamma_{N}} := \Big(\int_{\Gamma_{N}} | R(s) |^{2} ds + \int_{\Gamma_{N}} \int_{\Gamma_{N}} \frac{|R(s)-R(s')|^{2}}{|s-s'|^{d}}dsds'\Big)^{1/2}.$$
Let $\tau:\Omega \rightarrow \mathbb{R}$ be an extension of $T\lvert_{\Gamma_{N}} = 1$ into the domain such that $\| \tau\|_{1} \leq C_{tr} \| 1 \|_{1/2,\Gamma_{N}} = C_{tr} \lvert \Gamma_{N}\rvert^{1/2}$ for some $C_{tr}>0$.

\noindent \textbf{Remark:}  For natural convection within a unit square or cubic enclosure with a pair of differentially heated vertical walls, the linear conduction profile $\tau(x) = 1 - x_{1}$, where $x_{1}$ denotes the spatial coordinate in the horizontal direction, is such an extension satisfying: $\| \tau\|_{1} \leq \frac{2\sqrt{3}}{3}$.

The explicitly skew-symmetric trilinear forms are denoted:
\begin{align*}
b(u,v,w) &= \frac{1}{2} (u \cdot \nabla v, w) - \frac{1}{2} (u \cdot \nabla w, v) \; \; \; \forall u,v,w \in X, \\
b^{\ast}(u,T,S) &= \frac{1}{2} (u \cdot \nabla T, S) - \frac{1}{2} (u \cdot \nabla S, T) \; \; \; \forall u \in X, \; \forall T,S \in W.
\end{align*}
\noindent They enjoy the following continuity results and properties.
\begin{lemma} \label{l1}
There are constants $C_{1}, C_{2}, C_{3}, C_{4}, C_{5},$ and $C_{6}$ such that for all u,v,w $\in$ X and T,S $\in W$, $b(u,v,w)$ and $b^{\ast}(u,T,S)$ satisfy
\begin{align*}
b(u,v,w) &= (u \cdot \nabla v, w) + \frac{1}{2} ((\nabla \cdot u)v, w), \\
b^{\ast}(u,T,S) &= (u \cdot \nabla T, S) + \frac{1}{2} ((\nabla \cdot u)T, S), \\
b(u,v,w) &\leq C_{1} \| \nabla u \| \| \nabla v \| \| \nabla w \|, \\
b(u,v,w) &\leq C_{2} \sqrt{\| u \| \| \nabla u \|} \| \nabla v \| \| \nabla w \|, \\
b^{\ast}(u,T,S) &\leq C_{3} \| \nabla u \| \| \nabla T \| \| \nabla S \|,\\
b^{\ast}(u,T,S) &\leq C_{4} \sqrt{\| u \| \| \nabla u \|} \| \nabla T \| \| \nabla S \|,\\
b(u,v,w) &\leq C_{5} \| \nabla u \| \| \nabla v \| \sqrt{\| w \| \| \nabla w \|}, \\
b^{\ast}(u,T,S) &\leq C_{6} \| \nabla u \| \| \nabla T \| \sqrt{\| S \| \| \nabla S \|}.
\end{align*}
\begin{proof}
See Lemma 1 on p. 2 of \cite{Fiordilino}.
\end{proof}
\end{lemma}
The weak formulation of system (\ref{s2}) - (\ref{s2f}) is:
Find $u:[0,t^{\ast}] \rightarrow X$, $p:[0,t^{\ast}] \rightarrow Q$, $T:[0,t^{\ast}] \rightarrow W$ for a.e. $t \in (0,t^{\ast}]$ satisfying for $j = 1,...,J$:
\begin{align}
(u_{t},v) + b(u_,u,v) + Pr(\nabla u,\nabla v) - (p, \nabla \cdot v) &= PrRa(\xi T,v) + (f,v) \; \; \forall v \in X, \\
(q, \nabla \cdot u) &= 0 \; \; \forall q \in Q, \\
(T_{t},S) + b^{\ast}(u,T,S) + (\nabla T,\nabla S) &= (\gamma,S) \; \; \forall S \in W_{\Gamma_{1}}.
\end{align}
\subsection{Finite Element Preliminaries}
Consider a regular, quasi-uniform mesh $\Omega_{h} = \{K\}$ of $\Omega$ with maximum triangle diameter length $h$.  Let $X_{h} \subset X$, $Q_{h} \subset Q$, $\hat{W_{h}} = (W_{h},W_{\Gamma_{1},h}) \subset (W,W_{\Gamma_{1}}) = \hat{W}$ be conforming finite element spaces consisting of continuous piecewise polynomials of degrees \textit{j}, \textit{l}, and \textit{j}, respectively.  Moreover, assume they satisfy the following approximation properties $\forall 1 \leq j,l \leq k,m$:
\begin{align}
\inf_{v_{h} \in X_{h}} \Big\{ \| u - v_{h} \| + h\| \nabla (u - v_{h}) \| \Big\} &\leq Ch^{k+1} \lvert u \rvert_{k+1}, \label{a1}\\
\inf_{q_{h} \in Q_{h}}  \| p - q_{h} \| &\leq Ch^{m} \lvert p \rvert_{m}, \label{a2}\\
\inf_{S_{h} \in \hat{W_{h}}}  \Big\{ \| T - S_{h} \| + h\| \nabla (T - S_{h}) \| \Big\} &\leq Ch^{k+1} \lvert T \rvert_{k+1}, \label{a3}
\end{align}
for all $u \in X \cap H^{k+1}(\Omega)^{d}$, $p \in Q \cap H^{m}(\Omega)$, and $T \in \hat{W} \cap H^{k+1}(\Omega)$.  Furthermore, we consider those spaces for which the discrete inf-sup condition is satisfied,
\begin{equation} \label{infsup} 
\inf_{q_{h} \in Q_{h}} \sup_{v_{h} \in X_{h}} \frac{(q_{h}, \nabla \cdot v_{h})}{\| q_{h} \| \| \nabla v_{h} \|} \geq \beta > 0,
\end{equation}
\noindent where $\beta$ is independent of $h$.  Examples include the MINI-element, Taylor-Hood, and non-conforming Crouzeix-Raviart elements \cite{John}.  The space of discretely divergence free functions is defined by 
\begin{align*}
V_{h} := \{v_{h} \in X_{h} : (q_{h}, \nabla \cdot v_{h}) = 0, \forall q_{h} \in Q_{h}\}.
\end{align*}
The space $V_{h}^{\ast}$, dual to $V_{h}$, is endowed with the following dual norm
\begin{align*}
	\| w \|_{V_{h}^{\ast}} := \sup_{v_{h} \in V_{h}}\frac{(w,v_{h})}{\|\nabla v_{h}\|}.
\end{align*}
The discrete inf-sup condition implies that we may approximate functions in $V$ well by functions in $V_{h}$,
\begin{lemma} \label{l5}
	Suppose the discrete inf-sup condition (\ref{infsup}) holds, then for any $v \in V$
	\begin{equation*}
	\inf_{v_{h} \in V_{h}} \| \nabla (v - v_{h}) \| \leq C(\beta)\inf_{v_{h} \in X_{h}} \| \nabla (v - v_{h}) \|.
	\end{equation*}
\end{lemma}
\begin{proof}
	See Chapter 2, Theorem 1.1 on p. 59 of \cite{Girault}.
\end{proof}
We will also assume that the mesh and finite element spaces satisfy the standard inverse inequality \cite{Ern}:
$$\| \nabla \chi_{1,2} \| \leq C_{inv,1,2} h^{-1} \| \chi_{1,2} \| \; \; \; \forall \chi_{1} \in X_{h}, \; \forall \chi_{2} \in W_{\Gamma_{1},h},$$ 
\noindent where $C_{inv,1,2}$ depends on the minimum angle $\alpha_{min}$ in the triangulation.  A discrete Gronwall inequality will play a  role in the upcoming analysis.
\begin{lemma} \label{l4}
(Discrete Gronwall Lemma). Let $\Delta t$, H, $a_{n}$, $b_{n}$, $c_{n}$, and $d_{n}$ be finite nonnegative numbers for n $\geq$ 0 such that for N $\geq$ 1
\begin{align*}
a_{N} + \Delta t \sum^{N}_{0}b_{n} &\leq \Delta t \sum^{N-1}_{0} d_{n}a_{n} + \Delta t \sum^{N}_{0} c_{n} + H,
\end{align*}
then for all  $\Delta t > 0$ and N $\geq$ 1
\begin{align*}
a_{N} + \Delta t \sum^{N}_{0}b_{n} &\leq exp\big(\Delta t \sum^{N-1}_{0} d_{n}\big)\big(\Delta t \sum^{N}_{0} c_{n} + H\big).
\end{align*}
\end{lemma}
\begin{proof}
	See Lemma 5.1 on p. 369 of \cite{Heywood}.
\end{proof}
Lastly, the discrete time analysis will utilize the following norms $\forall \; 1 \leq k \leq \infty$:
\begin{align*}
\vertiii{v}_{\infty,k} &:= \max_{0\leq n \leq N} \| v^{n} \|_{k}, \;
\vertiii{v}_{p,k} := \big(\Delta t \sum^{N}_{n = 0} \| v^{n} \|^{p}_{k}\big)^{1/p}.
\end{align*}
\section{Numerical Scheme}
Denote the fully discrete solutions by $u^{n}_{h}$, $p^{n}_{h}$, and $T^{n}_{h}$ at time levels $t^{n} = n\Delta t$, $n = 1,2,...,N$, and $t^{\ast}=N\Delta t$.  Given $(u^{n-1}_{h}, p^{n-1}_{h}, T^{n-1}_{h})$ and $(u^{n}_{h}, p^{n}_{h}, T^{n}_{h})$ $\in (X_{h},Q_{h},W_{h})$, find $(u^{n+1}_{h}, p^{n+1}_{h}, T^{n+1}_{h})$ $\in (X_{h},Q_{h},W_{h})$ satisfying, for every $n = 1,2,...,N$, the fully discrete approximation of (\ref{s2}) - (\ref{s2f})
\begin{multline}\label{scheme:one:velocity}
(\frac{3u^{n+1}_{h} - 4u^{n}_{h} + u^{n-1}_{h}}{2\Delta t},v_{h}) + b(<u_{h}>^{n}_{e},u^{n+1}_{h},v_{h}) + b({u'}^{n}_{h},2u^{n}_{h} - u^{n-1}_{h},v_{h}) \\ + Pr(\nabla u^{n+1}_{h},\nabla v_{h}) - (p^{n+1}_{h}, \nabla \cdot v_{h}) =  PrRa(\xi (2T^{n}_{h}-T^{n-1}_{h}),v_{h}) + (f^{n+1},v_{h}) \; \; \forall v_{h} \in X_{h},
\end{multline}
\begin{equation}
(q_{h}, \nabla \cdot u^{n+1}_{h}) = 0 \; \; \forall q_{h} \in Q_{h},
\end{equation}
\begin{multline}\label{scheme:one:temperature}
(\frac{3T^{n+1}_{h} - 4T^{n}_{h} + T^{n-1}_{h}}{2\Delta t},S_{h}) + b^{\ast}(<u_{h}>^{n}_{e},T^{n+1}_{h},S_{h}) + b^{\ast}({u'}^{n}_{h},2T^{n}_{h} - T^{n-1}_{h},S_{h}) \\ + (\nabla T^{n+1}_{h},\nabla S_{h})  = (\gamma^{n+1},S_{h}) \; \; \forall S_{h} \in W_{\Gamma_{1},h}.
\end{multline}
\textbf{Remark:} To ensure second order accuracy of the method, the first iterate should be computed with a second order method such as the trapezoidal rule.\\
\textbf{Remark:} The treatment of the nonlinear terms in the time discretization (\ref{d1}) - (\ref{d2}) leads to a shared coefficient matrix, in the above, independent of the ensemble members.
\section{Numerical Analysis of the Ensemble Algorithm}
We present stability results for the aforementioned algorithm under the following timestep condition:
\begin{align}
\frac{C_{\dagger} \Delta t}{h} \max_{1 \leq j \leq J} \|\nabla {u'}^{n}_{h}\|^{2} \leq 1,\label{c1}
\end{align}
\noindent where $C_{\dagger} \equiv C_{\dagger}(|\Omega|, \alpha_{min},Pr)$.
In Theorem \ref{t1}, the nonlinear stability of the velocity, temperature, and pressure approximations are proven under condition \ref{c1} for the the scheme (\ref{scheme:one:velocity}) - (\ref{scheme:one:temperature}).  
\subsection{Stability Analysis}
\begin{theorem} \label{t1}
Suppose $f \in L^{\infty}(0,t^{\ast};H^{-1}(\Omega)^{d})$, $\gamma \in L^{\infty}(0,t^{\ast};H^{-1}(\Omega))$.  If the scheme (\ref{scheme:one:velocity}) - (\ref{scheme:one:temperature}) satisfies Condition \ref{c1}, then
\begin{multline}
\frac{1}{2} \|T^{N}_{h}\|^{2} + \frac{1}{2} \|2T^{N}_{h} - T^{N-1}_{h}\|^{2} + \|u^{N}_{h}\|^{2} + \|2u^{N}_{h} - u^{N-1}_{h}\|^{2} + \frac{1}{2} \sum_{n = 1}^{N-1} \|T^{n+1}_{h} - 2T^{n}_{h} + T^{n-1}_{h}\|^{2} 
\\ + \frac{1}{2} \sum_{n = 1}^{N-1} \|u^{n+1}_{h} - 2u^{n}_{h} + u^{n-1}_{h}\|^{2} + \frac{\Delta t}{2} \sum_{n = 1}^{N-1} \| \nabla T^{n+1}_{h} \|^{2} + Pr \Delta t \sum_{n = 1}^{N-1} \| \nabla u^{n+1}_{h} \|^{2}
\\ \leq exp(C t^{\ast}) \big\{\Delta t \sum_{n = 1}^{N-1} \Big( \frac{6}{Pr} \| f^{n+1} \|_{-1}^{2} + 4 \| \gamma^{n+1} \|_{-1}^{2} + {8 C_{I}^{2} C{tr}^{2} |\Gamma_{N}|} 
\\ + {6PrRa^{2}C_{PF,1}^2C_{I}^{2}C{tr}^{2} |\Gamma_{N}|}\Big) + 2\|T^{1}_{h}\|^{2} + 2\|2T^{1}_{h} - T^{0}_{h}\|^{2} + \|u^{1}_{h}\|^{2} + \|2u^{1}_{h} - u^{0}_{h}\|^{2} \big\}
\\ + C_{I}^{2}C_{tr}^{2}|\Gamma_{N}|\Big(2 + t^{\ast} + 4exp(C t^{\ast})\Big).
\end{multline}
\noindent Moreover,
\begin{multline}
\beta \Delta t \sum^{N-1}_{n=1} \| p^{n+1}_{h}\| \leq 2 \Big\{C_{1} \Delta t \| \nabla <u_{h}>^{n}_{e} \| \| \nabla u^{n+1}_{h} \| + \frac{2 C_{1} h}{C_{\dagger}} \Big( \| \nabla  u^{n}_{h} \| + \| \nabla  u^{n-1}_{h} \| \Big)
\\ + Pr \Delta t \|\nabla u^{n+1}_{h} \| +  2PrRaC_{PF,1} \Delta t \Big(\|2T^{n}_{h}-T^{n-1}_{h} \| + \frac{3C_{I} C_{tr} |\Gamma_{N}|^{1/2}}{2} \Big) + \Delta t \|f^{n+1} \|_{-1}\Big\}.
\end{multline}
\end{theorem}
\begin{proof}
Let $T^{n+1}_{h} = \theta^{n+1}_{h} + I_{h}\tau$, where $I_{h}\tau$ is an interpolant of $\tau$ satisfying $\|I_{h}\tau\|_{1} \leq C_{I} \|\tau\|_{1}$.  Add equations (\ref{scheme:one:velocity}) and (\ref{scheme:one:temperature}), let $(v_{h},q_{h},S_{h}) = (u^{n+1}_{h},p^{n+1}_{h},\theta^{n+1}_{h}) \in (V_{h},Q_{h},W_{\Gamma_{1},h})$ and use the polarization identity.  Then, 
\begin{multline}\label{stability:thicku}
\frac{1}{4 \Delta t} \Big\{\|\theta^{n+1}_{h}\|^{2} + \|2\theta^{n+1}_{h} - \theta^{n}_{h}\|^{2}\Big\} - \frac{1}{4 \Delta t} \Big\{\|\theta^{n}_{h}\|^{2} + \|2\theta^{n}_{h} - \theta^{n-1}_{h}\|^{2}\Big\} + \frac{1}{4 \Delta t} \|\theta^{n+1}_{h} - 2\theta^{n}_{h} + \theta^{n-1}_{h}\|^{2}
\\ + \frac{1}{4 \Delta t} \Big\{\|u^{n+1}_{h}\|^{2} + \|2u^{n+1}_{h} - u^{n}_{h}\|^{2}\Big\} - \frac{1}{4 \Delta t} \Big\{\|u^{n}_{h}\|^{2} + \|2u^{n}_{h} - u^{n-1}_{h}\|^{2}\Big\} + \frac{1}{4 \Delta t} \|u^{n+1}_{h} - 2u^{n}_{h} + u^{n-1}_{h}\|^{2}
\\ + \|\nabla \theta^{n+1}_{h}\|^{2} + (\nabla I_{h}\tau,\nabla \theta^{n+1}_{h}) + Pr \|\nabla u^{n+1}_{h}\|^{2} + b({u'}^{n}_{h},2u^{n}_{h} - u^{n-1}_{h},u^{n+1}_{h}) 
\\ + b^{\ast}({u'}^{n}_{h},2\theta^{n}_{h} - \theta^{n-1}_{h},\theta^{n+1}_{h}) = PrRa(\gamma (2\theta^{n}_{h}-\theta^{n-1}_{h} + I_{h}\tau), u^{n+1}_{h}) - b^{\ast}(u^{n}_{h},I_{h}\tau,\theta^{n+1}_{h})
\\ + (f^{n+1},u^{n+1}_{h}) + (\gamma^{n+1},T^{n+1}_{h}).
\end{multline}
Multiply by $\Delta t$, consider $-\Delta t (\nabla I_{h}\tau,\nabla \theta^{n+1}_{h})$ and $\Delta t PrRa(\xi I_{h}\tau, u^{n+1}_{h})$.  Use the Cauchy-Schwarz-Young inequality, interpolation estimates and note that $\| \xi \| = 1$,
\begin{align}
-\Delta t (\nabla I_{h}\tau,\nabla \theta^{n+1}_{h}) &\leq {2\Delta t} \|I_{h}\tau\|^{2}_{1} + \frac{\Delta t}{8} \|\nabla \theta^{n+1}_{h}\|^{2} \leq {2 C_{I}^{2} \Delta t} \|\tau\|^{2}_{1} + \frac{\Delta t}{8} \|\nabla \theta^{n+1}_{h}\|^{2}\label{stability:thick:estinttau}
\\ &\leq {2 C_{I}^{2} C{tr}^{2} |\Gamma_{N}| \Delta t} + \frac{\Delta t}{8} \|\nabla \theta^{n+1}_{h}\|^{2}, \notag
\\ \Delta t PrRa(\xi I_{h}\tau, u^{n+1}_{h}) &\leq \frac{\Delta t Pr^{2}Ra^{2}C_{PF,1}^2C_{I}^{2}C{tr}^{2} |\Gamma_{N}|}{2 \epsilon_{2}} + \frac{\Delta t \epsilon_{2}}{2} \| \nabla u^{n+1}_{h} \|^{2}. \label{stability:thick:esttau}
\end{align}
Use the Cauchy-Schwarz-Young inequality on $\Delta t (\gamma^{n+1},\theta^{n}_{h})$, $\Delta t PrRa(\xi (2\theta^{n}_{h}-\theta^{n-1}_{h}), u^{n+1}_{h})$, and \\ $\Delta t (f^{n+1},u^{n+1}_{h})$.  Then,
\begin{align}
\Delta t (\gamma^{n+1},\theta^{n+1}_{h}) &\leq {2\Delta t} \|\gamma^{n+1}\|^{2}_{-1} + \frac{\Delta t}{8} \|\nabla \theta^{n+1}_{h}\|^{2}, \label{stability:thick:estg}
\\ \Delta t PrRa(\xi (2\theta^{n}_{h}-\theta^{n-1}_{h}), u^{n+1}_{h}) &\leq \frac{\Delta t Pr^{2}Ra^{2}C_{PF,1}^2}{2 \epsilon_{1}} \|2\theta^{n}_{h}-\theta^{n-1}_{h} \|^{2} + \frac{\Delta t \epsilon_{1}}{2} \| \nabla u^{n+1}_{h} \|^{2}, \label{stability:thick:estT}
\\ \Delta t (f^{n+1},u^{n+1}_{h}) &\leq \frac{\Delta t}{2 \epsilon_{3}} \|f^{n+1} \|^{2}_{-1} + \frac{\Delta t \epsilon_{3}}{2} \| \nabla u^{n+1}_{h} \|^{2}. \label{stability:thick:estf}
\end{align}
Consider $-\Delta t b^{\ast}({u'}^{n}_{h},2\theta^{n}_{h} - \theta^{n-1}_{h},\theta^{n+1}_{h})$ and $\Delta t b({u'}^{n}_{h},2u^{n}_{h}-u^{n-1}_{h},u^{n+1}_{h})$.  Use skew-symmetry, Lemma \ref{l1}, the inverse inequality, and the Cauchy-Schwarz-Young inequality.  Then,
\begin{align}
-\Delta t b^{\ast}({u'}^{n}_{h},2\theta^{n}_{h} - \theta^{n-1}_{h},\theta^{n+1}_{h}) &= -\Delta t b^{\ast}({u'}^{n}_{h},\theta^{n+1}_{h},\theta^{n+1}_{h} - 2\theta^{n}_{h} + \theta^{n-1}_{h})  \label{stability:thick:estbstar}\\
&\leq \Delta t C_{6} \|\nabla {u'}^{n}_{h}\| \|\nabla \theta^{n+1}_{h}\| \sqrt{\|\theta^{n+1}_{h} - 2\theta^{n}_{h} + \theta^{n-1}_{h}\| \| \nabla (\theta^{n+1}_{h} - 2\theta^{n}_{h} + \theta^{n-1}_{h})\|} \notag
\\ &\leq  \frac{\Delta t C_{6} C^{1/2}_{inv,2}}{h^{1/2}} \|\nabla {u'}^{n}_{h}\| \|\nabla \theta^{n+1}_{h}\| \|\theta^{n+1}_{h} - 2\theta^{n}_{h} + \theta^{n-1}_{h}\| \notag
\\ &\leq \frac{2\Delta t^2 C_{6}^2 C_{inv,2}}{h} \|\nabla {u'}^{n}_{h}\|^{2} \|\nabla \theta^{n+1}_{h}\|^{2} + \frac{1}{8} \|\theta^{n+1}_{h} - 2\theta^{n}_{h} + \theta^{n-1}_{h}\|^{2}, \notag
\\ -\Delta t b({u'}^{n}_{h},2u^{n}_{h} - u^{n-1}_{h},u^{n+1}_{h}) &\leq \frac{2 \Delta t^{2} C_{5}^{2} C_{inv,1}}{h} \|\nabla {u'}^{n}_{h}\|^{2} \|\nabla u^{n+1}_{h}\|^{2}\notag
\\ &+ \frac{1}{8} \|u^{n+1}_{h} - 2u^{n}_{h} + u^{n-1}_{h}\|^{2}.\label{stability:thick:estb}
\end{align}
Use the Cauchy-Schwarz-Young and Poincar\'{e}-Friedrichs inequalities on $-\Delta t b^{\ast}(u^{n}_{h},I_{h}\tau,\theta^{n+1}_{h})$,
\begin{align}
-\Delta t b^{\ast}(u^{n}_{h},I_{h}\tau,\theta^{n+1}_{h}) &\leq \frac{1}{2} \| u^{n}_{h} \cdot \nabla I_{h} \tau \| \| \theta^{n+1}_{h} \| + \frac{1}{2} \| u^{n}_{h} \cdot \nabla \theta^{n+1}_{h} \| \| I_{h} \tau \| \label{stability:thick:estcouplingtau}
\\ &\leq \frac{(1 + C_{PF_{2}}C_{I}^{2}) C{tr}^{2} |\Gamma_{N}|\Delta t}{4 \epsilon_{4}} \| u^{n}_{h} \|^{2} + \frac{\epsilon_{4}}{4} \| \nabla \theta^{n+1}_{h} \|^{2}.\notag
\end{align}
\noindent Let $\epsilon_{1} = \epsilon_{2} = \epsilon_{3} = Pr/3$ and $\epsilon_{4} = 1$.  Using (\ref{stability:thick:estinttau}) - (\ref{stability:thick:estcouplingtau}) in (\ref{stability:thicku}) leads to
\begin{multline}
\frac{1}{4} \Big\{\|\theta^{n+1}_{h}\|^{2} + \|2\theta^{n+1}_{h} - \theta^{n}_{h}\|^{2}\Big\} - \frac{1}{4} \Big\{\|\theta^{n}_{h}\|^{2} + \|2\theta^{n}_{h} - \theta^{n-1}_{h}\|^{2}\Big\} + \frac{1}{8} \|\theta^{n+1}_{h} - 2\theta^{n}_{h} + \theta^{n-1}_{h}\|^{2} 
\\ + \frac{1}{4} \Big\{\|u^{n+1}_{h}\|^{2} + \|2u^{n+1}_{h} - u^{n}_{h}\|^{2}\Big\} - \frac{1}{4} \Big\{\|u^{n}_{h}\|^{2} + \|2u^{n}_{h} - u^{n-1}_{h}\|^{2}\Big\} + \frac{1}{8} \|u^{n+1}_{h} - 2u^{n}_{h} + u^{n-1}_{h}\|^{2}
\\ + \frac{\Delta t}{4} \|\nabla \theta^{n+1}_{h}\|^{2} + \frac{Pr \Delta t}{4} \|\nabla u^{n+1}_{h}\|^{2} + \frac{\Delta t}{2} \| \nabla \theta^{n+1}_{h} \|^{2} \big\{ 1- \frac{4 \Delta t C_{6}^2 C_{inv,2}}{h} \| \nabla {u'}^{n}_{h} \|^{2} \big\}
\\ + \frac{Pr \Delta t}{2} \| \nabla u^{n+1}_{h} \|^{2} \big\{ 1- \frac{4 \Delta t C_{5}^2 C_{inv,1}}{Pr h} \| \nabla {u'}^{n}_{h} \|^{2} \big\} \leq \frac{3\Delta tPr Ra^{2} C_{PF,1}^{2}}{2} \| 2\theta^{n}_{h}-\theta^{n-1}_{h} \|^{2}
\\ + \frac{(1 + C_{PF_{2}}C_{I}^{2}) C{tr}^{2} |\Gamma_{N}| \Delta t}{4} \| u^{n}_{h} \|^{2} + {2 C_{I}^{2} C{tr}^{2} |\Gamma_{N}| \Delta t} + \frac{3\Delta t PrRa^{2}C_{PF,1}^2C_{I}^{2}C{tr}^{2} |\Gamma_{N}|}{2}
\\ + \frac{3\Delta t}{2 Pr} \| f^{n+1} \|_{-1}^{2} + {\Delta t} \| \gamma^{n+1} \|_{-1}^{2}.
\end{multline}
\noindent \noindent Use the timestep condition \ref{c1}, multiply by 4, and add both $\|2u^{n}_{h} - u^{n-1}_{h}\|^{2}$ and $\|\theta^{n}_{h}\|^{2}$ to the r.h.s.  Taking a maximum over constants in the first two terms and the added terms on the r.h.s. and summing from $n = 1$ to $n = N-1$ leads to,
\begin{multline}
\|\theta^{N}_{h}\|^{2} + \|2\theta^{N}_{h} - \theta^{N-1}_{h}\|^{2} + \|u^{N}_{h}\|^{2} + \|2u^{N}_{h} - u^{N-1}_{h}\|^{2} + \frac{1}{2} \sum_{n = 1}^{N-1} \|\theta^{n+1}_{h} - 2\theta^{n}_{h} + \theta^{n-1}_{h}\|^{2}
\\ + \frac{1}{2} \sum_{n = 1}^{N-1} \|u^{n+1}_{h} - 2u^{n}_{h} + u^{n-1}_{h}\|^{2} + \Delta t \sum_{n = 1}^{N-1} \| \nabla \theta^{n+1}_{h} \|^{2} + Pr \Delta t \sum_{n = 1}^{N-1} \| \nabla u^{n+1}_{h} \|^{2} 
\\ \leq C \Delta t \sum_{n = 1}^{N-1} \big\{ \| u^{n}_{h} \|^{2} + \|2u^{n}_{h} - u^{n-1}_{h}\|^{2} + \| \theta^{n}_{h} \|^{2} + \|2\theta^{n}_{h} - \theta^{n-1}_{h}\|^{2} \big\} 
\\ + \Delta t \sum_{n = 1}^{N-1} \big\{ \frac{6}{Pr} \| f^{n+1} \|_{-1}^{2} + 4 \| \gamma^{n+1} \|_{-1}^{2} + {8 C_{I}^{2} C{tr}^{2} |\Gamma_{N}|} + {6PrRa^{2}C_{PF,1}^2C_{I}^{2}C{tr}^{2} |\Gamma_{N}|}{Pr}\big\}
\\ + \|\theta^{1}_{h}\|^{2} + \|2\theta^{1}_{h} - \theta^{0}_{h}\|^{2} + \|u^{1}_{h}\|^{2} + \|2u^{1}_{h} - u^{0}_{h}\|^{2}.
\end{multline}
Apply Lemma \ref{l4}.  Then,
\begin{multline}
\|\theta^{N}_{h}\|^{2} + \|2\theta^{N}_{h} - \theta^{N-1}_{h}\|^{2} + \|u^{N}_{h}\|^{2} + \|2u^{N}_{h} - u^{N-1}_{h}\|^{2} + \frac{1}{2} \sum_{n = 1}^{N-1} \|\theta^{n+1}_{h} - 2\theta^{n}_{h} + \theta^{n-1}_{h}\|^{2}
\\ + \frac{1}{2} \sum_{n = 1}^{N-1} \|u^{n+1}_{h} - 2u^{n}_{h} + u^{n-1}_{h}\|^{2} + \Delta t \sum_{n = 1}^{N-1} \| \nabla \theta^{n+1}_{h} \|^{2} + Pr \Delta t \sum_{n = 1}^{N-1} \| \nabla u^{n+1}_{h} \|^{2}
\\ \leq exp(C t^{\ast}) \big\{\Delta t \sum_{n = 1}^{N-1} \Big( \frac{6}{Pr} \| f^{n+1} \|_{-1}^{2} + 4 \| \gamma^{n+1} \|_{-1}^{2} + {8 C_{I}^{2} C{tr}^{2} |\Gamma_{N}|} + {6PrRa^{2}C_{PF,1}^2C_{I}^{2}C{tr}^{2} |\Gamma_{N}|}\Big)
\\ + \|\theta^{1}_{h}\|^{2} + \|2\theta^{1}_{h} - \theta^{0}_{h}\|^{2} + \|u^{1}_{h}\|^{2} + \|2u^{1}_{h} - u^{0}_{h}\|^{2}\big\}.
\end{multline}
\noindent The result follows by recalling the identity $T^{n+1}_{h} = \theta^{n+1}_{h} - I_{h}\tau$ and applying the triangle inequality.  Thus, numerical approximations of velocity and temperature are stable.  We now prove stability of the pressure approximation.  We first form an estimate for the discrete time derivative term.  Consider (\ref{scheme:one:velocity}), isolate $(\frac{3u^{n+1}_{h} - 4u^{n}_{h} + u^{n-1}_{h}}{2\Delta t},v_{h})$, let $0 \neq v_{h} \in V_{h}$, and multiply by $\Delta t$.  Then,
\begin{multline}\label{vv}
	\frac{1}{2} (3u^{n+1}_{h} - 4u^{n}_{h} + u^{n-1}_{h},v_{h}) = -\Delta t b(<u_{h}>^{n}_{e},u^{n+1}_{h},v_{h}) - \Delta t b({u'}^{n}_{h},2u^{n}_{h} - u^{n-1}_{h},v_{h}) 
	\\ - \Delta t Pr(\nabla u^{n+1}_{h},\nabla v_{h}) + \Delta t PrRa(\xi (2\theta^{n}_{h}-\theta^{n-1}_{h} + I_{h}\tau),v_{h}) + \Delta t (f^{n+1},v_{h}).
\end{multline}
Applying Lemma \ref{l1} to the skew-symmetric trilinear terms and the Cauchy-Schwarz and Poincar\'{e}-Friedrichs inequalities to the remaining terms yields
\begin{align}
-\Delta t b(<u_{h}>^{n}_{e},u^{n+1}_{h},v_{h}) &\leq C_{1} \Delta t \| \nabla <u_{h}>^{n}_{e} \| \| \nabla u^{n+1}_{h} \| \| \nabla v_{h}\|,\label{stability:thick:estmean}\\
-\Delta t b({u'}^{n}_{h},2u^{n}_{h} - u^{n-1}_{h},v_{h}) &\leq 2C_{1} \Delta t \| \nabla {u'}^{n}_{h} \| \Big\{ \| \nabla  u^{n}_{h} \| + \| \nabla  u^{n-1}_{h} \| \Big\} \| \nabla v_{h}\|,\label{stability:thick:estfluc}\\
-\Delta t Pr(\nabla u^{n+1}_{h},\nabla v_{h}) &\leq Pr \Delta t \|\nabla u^{n+1}_{h} \| \|\nabla v_{h} \|, \label{stability:thick:estvisc}\\
\Delta t PrRa(\xi (2\theta^{n}_{h}-\theta^{n-1}_{h}),v_{h}) &\leq PrRa \Delta t \|2\theta^{n}_{h}-\theta^{n-1}_{h} \| \|v_{h} \| \leq PrRaC_{PF,1} \Delta t \|2\theta^{n}_{h}-\theta^{n-1}_{h} \| \|\nabla v_{h} \|, \label{stability:thick:estcoupling}\\
\Delta t PrRa(\xi I_{h}\tau,v_{h}) &\leq PrRaC_{PF,1}C_{I} C_{tr}|\Gamma_{N}|^{1/2} \Delta t \|\nabla v_{h} \|, \label{stability:thick:esttaucoupling}\\
\Delta t(f^{n+1},v_{h}) &\leq \Delta t \|f^{n+1} \|_{-1} \|\nabla v_{h} \| \label{stability:thick:estf2}.
\end{align}
Apply the above estimates in (\ref{vv}), divide by the common factor $\|\nabla v_{h} \|$ on both sides, and take the supremum over all $0 \neq v_{h} \in V_{h}$.  Then,
\begin{multline}
\frac{1}{2}\|3 u^{n+1}_{h} - 4 u^{n}_{h} + u^{n-1}_{h} \|_{V^{\ast}_{h}} \leq C_{1} \Delta t \| \nabla <u_{h}>^{n}_{e} \| \| \nabla u^{n+1}_{h} \|
\\ + 2C_{1} \Delta t \| \nabla {u'}^{n}_{h} \| \Big\{ \| \nabla  u^{n}_{h} \| + \| \nabla  u^{n-1}_{h} \| \Big\} \| \nabla v_{h}\| + Pr \Delta t \|\nabla u^{n+1}_{h} \| +  PrRaC_{PF,1} \Delta t \|2\theta^{n}_{h}-\theta^{n-1}_{h}\|
\\ + PrRaC_{PF,1}C_{I}C_{tr}|\Gamma_{N}|^{1/2}  \Delta t + \Delta t \|f^{n+1} \|_{-1}.
\end{multline}
\noindent Reconsider equation (\ref{scheme:one:velocity}).  Multiply by $\Delta t$ and isolate the pressure term,
\begin{multline}
\Delta t (p^{n+1}_{h}, \nabla \cdot v_{h}) = \frac{1}{2}(3u^{n+1}_{h} - 4u^{n}_{h} + u^{n-1}_{h},v_{h}) + \Delta t b(<u_{h}>^{n}_{e},u^{n+1}_{h},v_{h}) + \Delta t b({u'}^{n}_{h},2u^{n}_{h} - u^{n-1}_{h},v_{h}) 
\\ + Pr \Delta t(\nabla u^{n+1}_{h},\nabla v_{h}) - PrRa \Delta t(\gamma (2\theta^{n}_{h}-\theta^{n-1}_{h} + I_{h}\tau),v_{h}) - \Delta t(f^{n+1},v_{h}).
\end{multline}
\noindent Apply (\ref{stability:thick:estmean}) - (\ref{stability:thick:estf2}) on the r.h.s terms.  Then,
\begin{multline}
\Delta t (p^{n+1}_{h}, \nabla \cdot v_{h}) \leq \frac{1}{2}(3u^{n+1}_{h} - 4u^{n}_{h} + u^{n-1}_{h},v_{h}) + \Big\{C_{1} \Delta t \| \nabla <u_{h}>^{n}_{e} \| \| \nabla u^{n+1}_{h} \|
\\ + 2 C_{1} \Delta t \| \nabla {u'}^{n}_{h} \| \Big( \| \nabla  u^{n}_{h} \| + \| \nabla  u^{n-1}_{h} \| \Big) + Pr \Delta t \|\nabla u^{n+1}_{h} \|
\\+  PrRaC_{PF,1} \Delta t \Big(\|2\theta^{n}_{h}-\theta^{n-1}_{h} \| + C_{I}C_{tr}|\Gamma_{N}|^{1/2} \Big) + \Delta t \|f^{n+1} \|_{-1}\Big\}\|\nabla v_{h} \|.
\end{multline}
\noindent Divide by $\|\nabla v_{h} \|$ and note that $\frac{(3u^{n+1}_{h} - 4u^{n}_{h} + u^{n-1}_{h},v_{h})}{2\| \nabla v_{h} \|} \leq \frac{1}{2}\|3u^{n+1}_{h} - 4u^{n}_{h} + u^{n-1}_{h}\|_{V^{\ast}_{h}}$.  Take the supremum over all $0 \neq v_{h} \in X_{h}$,
\begin{multline}
\Delta t \sup_{0 \neq v_{h} \in X_{h}}\frac{(p^{n+1}_{h}, \nabla \cdot v_{h})}{\|\nabla v_{h} \|} \leq 2 \Big\{C_{1} \Delta t \| \nabla <u_{h}>^{n}_{e} \| \| \nabla u^{n+1}_{h} \| + 2 C_{1} \Delta t \| \nabla {u'}^{n}_{h} \| \Big( \| \nabla  u^{n}_{h} \| + \| \nabla  u^{n-1}_{h} \| \Big)
\\ + Pr \Delta t \|\nabla u^{n+1}_{h} \| +  PrRaC_{PF,1} \Delta t \Big(\|2\theta^{n}_{h}-\theta^{n-1}_{h} \| + C_{I}C_{tr}|\Gamma_{N}|^{1/2} \Big) + \Delta t \|f^{n+1} \|_{-1}\Big\}.
\end{multline}
\noindent Use the inf-sup condition,
\begin{multline}
\beta \Delta t \| p^{n+1}_{h}\| \leq 2 \Big\{C_{1} \Delta t \| \nabla <u_{h}>^{n}_{e} \| \| \nabla u^{n+1}_{h} \| + 2 C_{1} \Delta t \| \nabla {u'}^{n}_{h} \| \Big( \| \nabla  u^{n}_{h} \| + \| \nabla  u^{n-1}_{h} \| \Big)
\\ + Pr \Delta t \|\nabla u^{n+1}_{h} \| +  PrRaC_{PF,1} \Delta t \Big(\|2\theta^{n}_{h}-\theta^{n-1}_{h} \| + C_{I}C_{tr}|\Gamma_{N}|^{1/2} \Big) + \Delta t \|f^{n+1} \|_{-1}\Big\}.
\end{multline}
\noindent Sum from $n = 1$ to $n = N-1$, use condition \ref{c1}, recall $T^{n+1}_{h} = \theta^{n+1}_{h} + I_{h}\tau$, and use the triangle inequality.  The result follows, yielding stability of the pressure approximation, built on the stability of the temperature and velocity approximations.
\end{proof}
\textbf{Remark:}  Application of Lemma \ref{l4} in Theorem \ref{t1} allows for the loss of long time stability due to the exponential growth factor, in $t^{\ast}$.
\subsection{Error Analysis}
Denote $u^{n}$, $p^{n}$, and $T^{n}$ as the true solutions at time $t^{n} = n\Delta t$.  Assume the solutions satisfy the following regularity assumptions:
\begin{align} 
u &\in L^{\infty}(0,t^{\ast};X \cap H^{k+1}(\Omega)), \;  T, \tau \in L^{\infty}(0,t^{\ast};W \cap H^{k+1}(\Omega)), \notag \\
u_{t}, T_{t} &\in L^{\infty}(0,t^{\ast};H^{k+1}(\Omega)), \; u_{tt}, T_{tt} \in L^{\infty}(0,t^{\ast};H^{k+1}(\Omega)), \label{error:regularity} \\
u_{ttt}, T_{ttt} &\in L^{\infty}(0,t^{\ast};H^{k+1}(\Omega)), \; p \in L^{\infty}(0,t^{\ast};Q \cap H^{m}(\Omega)). \notag
\end{align}
\noindent \textbf{Remark:}  Regularity of the auxiliary temperature solution $\theta$ follows from the above regularity assumptions.  Convergence results will be proven for the error in the auxiliary variable $\theta$ which, by the triangle inequality and interpolation estimates, implies the results for the solution variable $T$.  

The errors for the solution variables are denoted
\begin{align*}
e^{n}_{u} &= u^{n} - u^{n}_{h}, \; e^{n}_{T} = T^{n} - T^{n}_{h}, \; e^{n}_{p} = p^{n} - p^{n}_{h}.
\end{align*}\vspace*{-\baselineskip}
\begin{definition} (Consistency error).  The consistency errors are defined as
\begin{align*}
\varsigma_{u}(u^{n};v_{h}) = \big(\frac{3u^{n+1} - 4u^{n} + u^{n-1}}{2\Delta t} - u^{n+1}_{t}, v_{h}\big), \; \varsigma_{T}(T^{n};S_{h}) = \big(\frac{3T^{n+1} - 4T^{n} + T^{n-1}}{2\Delta t} - T^{n+1}_{t}, S_{h}\big).
\end{align*}
\end{definition}
\begin{lemma}\label{consistency}
Provided  $u$ and $T$ satisfy the regularity assumptions \ref{error:regularity}, then $\exists \; C>0$ such that $\forall \; r > 0$
\begin{align*}
\lvert \varsigma_{u}(u^{n};v_{h}) \rvert &\leq \frac{C C^{2}_{PF,1} C_{r} \Delta t^{3}}{\epsilon}\| u_{ttt}\|^{2}_{L^{2}(t^{n-2},t^{n};L^{2}(\Omega))} + \frac{\epsilon}{r} \| \nabla v_{h} \|^{2},
\\ \lvert \varsigma_{T}(T^{n};S_{h}) \rvert &\leq \frac{C C^{2}_{PF,2} C_{r} \Delta t^{3}}{\epsilon}\| T_{ttt}\|^{2}_{L^{2}(t^{n-2},t^{n};L^{2}(\Omega))} + \frac{\epsilon}{r} \| \nabla S_{h} \|^{2}.
\end{align*}
\end{lemma}
\begin{proof}
These follow from the Cauchy-Schwarz-Young inequality, Poincar\'{e}-Friedrichs inequality, and Taylor's Theorem with integral remainder.
\end{proof}
\begin{theorem} \label{error:thick}
For (u,p,T) satisfying (1) - (5), suppose that $(u^{0}_{h},p^{0}_{h},T^{0}_{h}) \in (X_{h},Q_{h},W_{h})$ are approximations of $(u^{0},p^{0},T^{0})$ to within the accuracy of the interpolant.  Further, suppose that condition \ref{c1} holds. Then there exists a constant C such that

\begin{multline*}
\frac{1}{2}\|e^{N}_{T}\|^{2} + \frac{1}{2}\|2e^{N}_{T}-e^{N-1}_{T}\|^{2} + \|e^{N}_{u}\|^{2} + \|2e^{N}_{u} - e^{N-1}_{u}\|^{2} + \frac{1}{2}\sum_{n = 1}^{N-1}\big(\|e^{n+1}_{T} - 2e^{n}_{T} + e^{n-1}_{T}\|^{2} + \|e^{n+1}_{u} - 2e^{n}_{u} + e^{n-1}_{u}\|^{2}\big)
\\ + \frac{\Delta t}{2} \sum_{n = 1}^{N-1} \|\nabla e^{n+1}_{T}\|^{2} + Pr \Delta t \sum_{n = 1}^{N-1} \|\nabla e^{n+1}_{u}\|^{2} + \frac{Pr \Delta t}{2} \Big(\|\nabla e^{N}_{u}\|^{2} +\frac{1}{2} \|\nabla e^{N-1}_{u}\|^{2}\Big)
\\ \leq \exp(Ct^{\ast}) \Big\{ \Delta t \inf_{S_{h} \in W_{\Gamma_{1,h}}} \Big( \vertiii{\theta - S_{h}}^{2}_{\infty,0} + \vertiii{\theta - S_{h}}_{\infty,0} \vertiii{\nabla (\theta - S_{h})}_{\infty,0} + \vertiii{\nabla (\theta - S_{h})}^{2}_{\infty,0} + \vertiii{(\theta - S_{h})_{t}}^{2}_{\infty,0}
\\ + h \Delta t^{2} \vertiii{(\theta - S_{h})_{tt}}^{2}_{\infty,0} \Big) + \Delta t \inf_{v_{h} \in X_{h}} \Big( \vertiii{u - v_{h} }^{2}_{\infty,0} + \vertiii{u - v_{h} }^{2}_{\infty,0} \vertiii{\nabla (u - v_{h}) }_{\infty,0} + \vertiii{\nabla (u - v_{h}) }^{2}_{\infty,0}
\\ + \vertiii{(u - v_{h})_{t}}^{2}_{\infty,0} + h \Delta t^{2} \vertiii{(u - v_{h})_{tt}}^{2}_{\infty,0} \Big) + \Delta t \inf_{q_{h} \in Q_{h}} \vertiii{ p - q_{h} }^{2}_{\infty,0}
\\ + \Delta t \inf_{S_{h} \in W_{h}} \Big( \vertiii{\tau - I_{h}\tau}^{2}_{\infty,0} + \vertiii{\nabla (\tau - I_{h}\tau) }^{2}_{\infty,0}\Big) + h\Delta t^{3} + \Delta t^{4} \Big\}
\\ + \|\zeta^{0}_{h}\|^{2} + \|2\zeta^{1}_{h}-\zeta^{0}_{h}\|^{2} + \|\eta^{0}_{h}\|^{2} + \|2\eta^{1}_{h} - \phi^{0}_{h}\|^{2} + \frac{Pr \Delta t}{2} \Big(\|\nabla \eta^{1}_{h}\|^{2} +\frac{1}{2} \|\nabla \eta^{0}_{h}\|^{2}\Big)
\\ + \frac{1}{2}\Big(\|e^{0}_{T}\|^{2} + \|2e^{1}_{T}-e^{0}_{T}\|^{2}\Big) + \|e^{0}_{u}\|^{2} + \|2e^{1}_{u} - e^{0}_{u}\|^{2} + \frac{Pr \Delta t}{2} \Big(\|\nabla e^{1}_{u}\|^{2} +\frac{1}{2} \|\nabla e^{0}_{u}\|^{2}\Big).
\end{multline*}

\end{theorem}
\begin{proof}
Let $T = \theta + \tau$.  The true solutions satisfy for all $n = 1, ... N-1$:
\begin{align}
(\frac{3u^{n+1} - 4u^{n} + u^{n-1}}{2\Delta t},v_{h}) + b(u^{n+1},u^{n+1},v_{h}) + Pr(\nabla u^{n+1},\nabla v_{h}) - (p^{n+1}, \nabla \cdot v_{h}) \label{error:one:truevelocity}
\\ =  PrRa(\gamma (\theta^{n+1}+\tau),v_{h}) + (f^{n+1},v_{h}) + \varsigma_{u}(u^{n+1};v_{h}) \; \; \forall v_{h} \in X_{h} \notag,
\\ (q_{h}, \nabla \cdot u^{n+1}) = 0 \; \; \forall q_{h} \in Q_{h},
\\ (\frac{3\theta^{n+1} - 4\theta^{n} + \theta^{n-1}}{2\Delta t},S_{h}) + b^{\ast}(u^{n+1},\theta^{n+1},S_{h}) + (\nabla \theta^{n+1},\nabla S_{h}) + (\nabla \tau,\nabla S_{h}) \label{error:one:truetemp}
\\ = (\gamma^{n+1},S_{h}) + \varsigma_{T}(\theta^{n+1};S_{h}) \; \; \forall S_{h} \in W_{\Gamma_{1,h}}.\notag
\end{align}
Subtract (\ref{error:one:truetemp}) and (\ref{scheme:one:temperature}), then the error equation for temperature is
\begin{align}
(\frac{3e^{n+1}_{\theta} - 4e^{n}_{\theta} + e^{n-1}_{\theta}}{2\Delta t},S_{h}) + b^{\ast}(u^{n+1},\theta^{n+1},S_{h}) - b^{\ast}(<u_{h}>^{n}_{e},\theta^{n+1}_{h},S_{h}) - b^{\ast}({u'}^{n}_{h},2\theta^{n}_{h}-\theta^{n-1}_{h},S_{h})
\\ + b^{\ast}(u^{n+1},\tau,S_{h})-b^{\ast}(u^{n}_{h},I_{h}\tau,S_{h}) + (\nabla e^{n+1}_{\theta},\nabla S_{h}) + (\nabla (\tau - I_{h}\tau),\nabla S_{h})  = \varsigma_{T}(\theta^{n+1},S_{h}) \; \; \forall S_{h} \in W_{\Gamma_{1,h}} \notag.
\end{align}
Letting $e^{n}_{\theta} = (\theta^{n} - \tilde{\theta}^{n}) - (\theta^{n}_{h}- \tilde{\theta}^{n}) = \zeta^{n} - \psi^{n}_{h}$  and rearranging give,
\begin{multline*}
(\frac{3\psi^{n+1}_{h} - 4\psi^{n}_{h} + \psi^{n-1}_{h}}{2\Delta t},S_{h}) + (\nabla \psi^{n+1}_{h},\nabla S_{h}) = (\frac{3\zeta^{n+1} - 4\zeta^{n} + \zeta^{n-1}}{2\Delta t},S_{h}) + (\nabla \zeta^{n+1},\nabla S_{h})
\\ + (\nabla (\tau - I_{h}\tau),\nabla S_{h}) + b^{\ast}(u^{n+1},\theta^{n+1},S_{h}) - b^{\ast}(2u^{n}_{h}-u^{n-1}_{h},\theta^{n+1}_{h},S_{h}) - b^{\ast}({u'}^{n}_{h},-\theta^{n+1}_{h}+2\theta^{n}_{h}-\theta^{n-1}_{h},S_{h})
\\ + b^{\ast}(u^{n+1},\tau,S_{h})-b^{\ast}(u^{n}_{h},I_{h}\tau,S_{h}) - \varsigma_{T}(\theta^{n+1},S_{h}) \; \; \forall S_{h} \in W_{\Gamma_{1,h}}.
\end{multline*}
Setting $S_{h} = \psi^{n+1}_{h} \in W_{\Gamma_{1,h}}$ yields
\begin{multline*}
\frac{1}{4 \Delta t} \Big\{\|\psi^{n+1}_{h}\|^{2} + \|2\psi^{n+1}_{h}-\psi^{n}_{h}\|^{2}\Big\} - \frac{1}{4 \Delta t} \Big\{\|\psi^{n}_{h}\|^{2} + \|2\psi^{n}_{h} - \psi^{n-1}_{h}\|^{2}\Big\} + \frac{1}{4 \Delta t} \|\psi^{n+1}_{h} - 2\psi^{n}_{h} + \psi^{n-1}_{h}\|^{2} + \|\nabla \psi^{n+1}_{h}\|^{2}
\\ = \frac{1}{2\Delta t}(3\zeta^{n+1}-4\zeta^{n}+\zeta^{n-1},\psi^{n+1}_{h}) + (\nabla \zeta^{n+1},\nabla \psi^{n+1}_{h}) + (\nabla (\tau - I_{h}\tau),\nabla \psi^{n+1}_{h}) + b^{\ast}(u^{n+1},\theta^{n+1},\psi^{n+1}_{h})
\\ - b^{\ast}(2u^{n}_{h}-u^{n-1}_{h},\theta^{n+1}_{h},\psi^{n+1}_{h}) - b^{\ast}({u'}^{n}_{h},-\theta^{n+1}_{h}+2\theta^{n}_{h}-\theta^{n-1}_{h},\psi^{n+1}_{h}) + b^{\ast}(u^{n+1},\tau,\psi^{n+1}_{h})
\\- b^{\ast}(u^{n}_{h},I_{h}\tau,\psi^{n+1}_{h}) - \varsigma_{T}(\theta^{n+1},\psi^{n+1}_{h}).
\end{multline*}
Add and subtract $b^{\ast}(u^{n+1},\theta^{n+1}_{h},\psi^{n+1}_{h})$, $b^{\ast}(2u^{n}-u^{n-1},\theta^{n+1}_{h},\psi^{n+1}_{h})$, $b^{\ast}({u'}^{n}_{h},-\theta^{n+1}+2\theta^{n}-\theta^{n-1},\psi^{n+1}_{h})$, and $b^{\ast}(2u^{n}-u^{n-1},\tau-I_{h}\tau,\psi^{n+1}_{h})$ to the r.h.s.  Rearrange, then
\begin{multline}\label{fet1}
\frac{1}{4 \Delta t} \Big\{\|\psi^{n+1}_{h}\|^{2} + \|2\psi^{n+1}_{h}-\psi^{n}_{h}\|^{2}\Big\} - \frac{1}{4 \Delta t} \Big\{\|\psi^{n}_{h}\|^{2} + \|2\psi^{n}_{h} - \psi^{n-1}_{h}\|^{2}\Big\} + \frac{1}{4 \Delta t} \|\psi^{n+1}_{h} - 2\psi^{n}_{h} + \psi^{n-1}_{h}\|^{2}
\\ + \|\nabla \psi^{n+1}_{h}\|^{2} = \frac{1}{\Delta 2t}(3\zeta^{n+1}-4\zeta^{n}+\zeta^{n-1},\psi^{n+1}_{h}) + (\nabla \zeta^{n+1},\nabla \psi^{n+1}_{h}) + b^{\ast}(u^{n+1},\zeta^{n+1},\psi^{n+1}_{h}) 
\\ + b^{\ast}(u^{n+1}-2u^{n}+u^{n-1},\theta^{n+1}_{h},\psi^{n+1}_{h}) + b^{\ast}(2\eta^{n}-\eta^{n-1},\theta^{n+1}_{h},\psi^{n+1}_{h})
\\ - b^{\ast}(2\phi^{n}_{h}-\phi^{n-1}_{h},\theta^{n+1}_{h},\psi^{n+1}_{h}) - b^{\ast}({u'}^{n}_{h},\zeta^{n+1}-2\zeta^{n}+\zeta^{n-1},\psi^{n+1}_{h}) + b^{\ast}({u'}^{n}_{h},\psi^{n+1}_{h}-2\psi^{n}_{h}+\psi^{n-1}_{h},\psi^{n+1}_{h})
\\ + b^{\ast}({u'}^{n}_{h},\theta^{n+1}-2\theta^{n}+\theta^{n-1},\psi^{n+1}_{h}) + b^{\ast}(u^{n+1}-2u^{n}+u^{n-1},\tau,\psi^{n+1}_{h}) + b^{\ast}(2u^{n}-u^{n-1},\tau - I_{h}\tau,\psi^{n+1}_{h})
\\ + b^{\ast}(2\eta^{n}-\eta^{n-1},I_{h}\tau,\psi^{n+1}_{h}) - b^{\ast}(2\phi^{n}_{h}-\phi^{n-1}_{h},I_{h}\tau,\psi^{n+1}_{h}) + (\nabla (\tau - I_{h}\tau),\nabla \psi^{n+1}_{h}) - \varsigma_{T}(T^{n+1},\psi^{n+1}_{h}).
\end{multline}
Follow analogously for the velocity error equation.  Subtract (\ref{error:one:truevelocity}) and (\ref{scheme:one:velocity}), split the error into $e^{n}_{u} = (u^{n} - \tilde{u}^{n}) - (u^{n}_{h}- \tilde{u}^{n}) = \eta^{n} - \phi^{n}_{h}$, let $v_{h} = \phi^{n+1}_{h} \in V_{h}$, add and subtract $PrRa(\xi(2\theta^{n} - \theta^{n-1} + \tau),\phi^{n+1}_{h})$, $b(u^{n+1},u^{n+1}_{h},\phi^{n+1}_{h})$, $b(2u^{n}-u^{n-1},u^{n+1}_{h},\phi^{n+1}_{h})$, and $b({u'}^{n}_{h},-u^{n+1}+2u^{n}-u^{n-1},\phi^{n+1}_{h})$.  Then,
\begin{multline}\label{feu1}
\frac{1}{4 \Delta t} \Big\{\|\phi^{n+1}_{h}\|^{2} + \|2\phi^{n+1}_{h}-\phi^{n}_{h}\|^{2}\Big\} - \frac{1}{4 \Delta t} \Big\{\|\phi^{n}_{h}\|^{2} + \|2\phi^{n}_{h} - \phi^{n-1}_{h}\|^{2}\Big\} + \frac{1}{4 \Delta t} \|\phi^{n+1}_{h} - 2\phi^{n}_{h} + \phi^{n-1}_{h}\|^{2}
\\ + Pr \|\nabla \phi^{n+1}_{h}\|^{2} = \frac{1}{2\Delta t}(3\eta^{n+1}-4\eta^{n}+\eta^{n-1},\phi^{n+1}_{h}) + Pr (\nabla \eta^{n+1},\nabla \phi^{n+1}_{h}) - (p^{n+1} - q^{n+1}_{h},\nabla \cdot \phi^{n+1}_{h})
\\ - PrRa(\xi (\theta^{n+1}-2\theta^{n}+\theta^{n-1}),\phi^{n+1}_{h}) - PrRa(\xi (2\zeta^{n} - \zeta^{n-1}),\phi^{n+1}_{h}) + PrRa(\xi (2\psi^{n}_{h}-\psi^{n-1}_{h}),\phi^{n+1}_{h})
\\ - PrRa(\xi (\tau - I_{h}\tau),\phi^{n+1}_{h}) + b(u^{n+1},\eta^{n+1},\phi^{n+1}_{h}) + b(u^{n+1}-2u^{n}+u^{n-1},u^{n+1}_{h},\phi^{n+1}_{h})
\\ + b(2\eta^{n}-\eta^{n-1},u^{n+1}_{h},\phi^{n+1}_{h}) - b(2\phi^{n}_{h}-\phi^{n-1}_{h},u^{n+1}_{h},\phi^{n+1}_{h}) - b({u'}^{n}_{h},\eta^{n+1}-2\eta^{n}+\eta^{n-1},\phi^{n+1}_{h})
\\ + b({u'}^{n}_{h},\phi^{n+1}_{h}-2\phi^{n}_{h}+\phi^{n-1}_{h},\phi^{n+1}_{h}) + b({u'}^{n}_{h},u^{n+1}-2u^{n}+u^{n-1},\phi^{n+1}_{h}) - \varsigma_{u}(u^{n+1},\phi^{n+1}_{h}).
\end{multline}
We seek to now estimate all terms on the r.h.s. in such a way that we may subsume the terms involving unknown pieces $\psi^{k}_{h}$ and $\phi^{k}_{h}$ into the l.h.s.  The following estimates are formed using skew-symmetry, Lemma \ref{l1}, and the Cauchy-Schwarz-Young inequality,
\begin{align}
b^{\ast}(u^{n+1},\zeta^{n+1},\psi^{n+1}_{h}) &\leq C_{6} \| \nabla u^{n+1} \| \| \nabla \psi^{n+1}_{h} \| \sqrt{\| \zeta^{n+1} \| \| \nabla \zeta^{n+1} \|}  
\\ &\leq \frac{C_{r} C_{6}^{2}}{\epsilon_3} \| \nabla u^{n+1} \|^{2} \| \| \zeta^{n+1} \| \| \nabla \zeta^{n+1} \| + \frac{\epsilon_3}{r} \| \nabla \psi^{n+1}_{h} \|^{2},\notag 
\\ b^{\ast}(2\eta^{n}-\eta^{n-1},\theta^{n+1}_{h},\psi^{n+1}_{h}) &\leq C_{4} \| \nabla \theta^{n+1}_{h} \| \| \nabla \psi^{n+1}_{h} \|\Big\{2\sqrt{\| \eta^{n}\| \| \nabla \eta^{n}\|} + \sqrt{\| \eta^{n-1}\| \| \nabla \eta^{n-1}\|}\}  
\\ &\leq \frac{8C_{r} C_{4}^{2}}{\epsilon_5} \| \nabla \theta^{n+1}_{h} \|^{2} \Big\{\| \eta^{n}\| \| \nabla \eta^{n}\| + \| \eta^{n-1}\| \| \nabla \eta^{n-1}\|\Big\} + \frac{\epsilon_5}{r} \| \nabla \psi^{n+1}_{h} \|^{2}. \notag
\end{align}
Applying Lemma \ref{l1}, the Cauchy-Schwarz-Young inequality, and Taylor's theorem yields,
\begin{align}
b^{\ast}(u^{n+1}-2u^{n}+u^{n-1},\theta^{n+1}_{h},\psi^{n+1}_{h}) &\leq C_{3} \| \nabla (u^{n+1}-2u^{n}+u^{n-1}) \| \| \nabla \theta^{n+1}_{h} \| \| \nabla \psi^{n+1}_{h} \| 
\\ &\leq \frac{C_{r} C_{3}^{2}}{\epsilon_4} \| \nabla (u^{n+1}-2u^{n}+u^{n-1}) \|^{2} \| \nabla \theta^{n+1}_{h} \|^{2} + \frac{\epsilon_4}{r} \| \nabla \psi^{n+1}_{h} \|^{2} \notag
\\ &\leq \frac{C C_{r} C_{3}^{2} \Delta t^{3}}{\epsilon_4} \| \nabla \theta^{n+1}_{h} \|^{2} \| \nabla u_{tt} \|^{2}_{L^{2}(t^{n-1},t^{n+1};L^{2}(\Omega))} + \frac{\epsilon_4}{r} \| \nabla \psi^{n+1}_{h} \|^{2},\notag
\\ -b^{\ast}({u'}^{n}_{h},\zeta^{n+1}-2\zeta^{n}+\zeta^{n-1},\psi^{n+1}_{h}) &\leq C_{3} \| \nabla {u'}^{n}_{h} \| \| \nabla \psi^{n+1}_{h} \|\| \nabla (\zeta^{n+1}-2\zeta^{n}+\zeta^{n-1}) \|
\\ &\leq \frac{C C_{r} C_{3}^{2} \Delta t^{3}}{\epsilon_7} \| \nabla {u'}^{n}_{h} \|^{2}\| \nabla \zeta_{tt} \|^{2}_{L^{2}(t^{n-1},t^{n+1};L^{2}(\Omega))} + \frac{\epsilon_7}{r} \| \nabla \psi^{n+1}_{h} \|^{2}, \notag
\\ b^{\ast}({u'}^{n}_{h},\theta^{n+1}-2\theta^{n}+\theta^{n-1},\psi^{n+1}_{h}) &\leq C_{3} \| \nabla {u'}^{n}_{h} \| \| \nabla (\theta^{n+1}-2\theta^{n}+\theta^{n-1}) \| \| \nabla \psi^{n+1}_{h} \| 
\\ &\leq \frac{C C_{r} C_{3}^{2}\Delta t^{3}}{\epsilon_{9}} \| \nabla {u'}^{n}_{h} \|^{2} \| \nabla \theta_{tt} \|^{2}_{L^{2}(t^{n-1},t^{n+1};L^{2}(\Omega))} + \frac{\epsilon_{9}}{r} \| \nabla \psi^{n+1}_{h}\|^{2}.\notag
\end{align}
Apply the triangle inequality, Lemma \ref{l1} and the Cauchy-Schwarz-Young inequality twice.  This yields
\begin{align}
-b^{\ast}(2\phi^{n}_{h}-\phi^{n-1}_{h},\theta^{n+1}_{h},\psi^{n+1}_{h}) &\leq C_{4}\| \nabla \theta^{n+1}_{h} \| \| \nabla \psi^{n+1}_{h} \| \sqrt{\| 2\phi^{n}_{h}-\phi^{n-1}_{h} \| \| \nabla (2\phi^{n}_{h}-\phi^{n-1}_{h}) \|} 
\\ &\leq C_{4} C_{\theta}(j) \| \nabla \psi^{n+1}_{h} \| \sqrt{\| 2\phi^{n}_{h}-\phi^{n-1}_{h} \| \| \nabla (2\phi^{n}_{h}-\phi^{n-1}_{h}) \|} \notag
\\ &\leq \epsilon_{6}\| \nabla \psi^{n+1}_{h} \|^2 + \frac{C_{4}^{2} C_{\theta}^{2}}{4\epsilon_{6}} \| 2\phi^{n}_{h}-\phi^{n-1}_{h} \| \| \nabla (2\phi^{n}_{h}-\phi^{n-1}_{h}) \| \notag
\\ &\leq \epsilon_{6}\| \nabla \psi^{n+1}_{h} \|^2 + \frac{C_{4}^{2} C_{\theta}^{2}}{8\epsilon_{6}\delta_{6}} \| 2\phi^{n}_{h}-\phi^{n-1}_{h} \|^{2} \notag
\\ &+ \frac{C_{4}^{2} C_{\theta}^{2}\delta_{6}}{2\epsilon_{6}} \Big( \| \nabla \phi^{n}_{h} \|^{2} + \| \nabla \phi^{n-1}_{h} \|^{2}\Big), \notag
\\ - b^{\ast}(2\phi^{n}_{h}-\phi^{n-1}_{h},I_{h}\tau,\psi^{n+1}_{h}) &\leq \epsilon_{13}\| \nabla \psi^{n+1}_{h} \|^2 + \frac{C_{4}^{2} C_{I}^{2}C_{tr}^{2} |\Gamma_{N}|}{8\epsilon_{13}\delta_{13}} \| 2\phi^{n}_{h}-\phi^{n-1}_{h} \|^{2}
\\ &+ \frac{C_{4}^{2} C_{I}^{2}C_{tr}^{2} |\Gamma_{N}|\delta_{13}}{2\epsilon_{13}} \Big( \| \nabla \phi^{n}_{h} \|^{2} + \| \nabla \phi^{n-1}_{h} \|^{2}\Big). \notag
\end{align}
Use Lemma \ref{l1}, the inverse inequality, and the Cauchy-Schwarz-Young inequality yielding
\begin{align}
\Delta t b^{\ast}({u'}^{n}_{h},\psi^{n+1}_{h}-2\psi^{n}_{h}+\psi^{n-1}_{h},\psi^{n+1}_{h}) &\leq \frac{\Delta t C_{6} C^{1/2}_{inv,2}}{h^{1/2}} \| \nabla {u'}^{n}_{h} \| \| \nabla \psi^{n+1} \| \| \psi^{n+1}_{h}-2\psi^{n}_{h}+\psi^{n-1}_{h} \|
\\ &\leq \frac{2C_{6}^{2} C_{inv,2} \Delta t^2}{h} \| \nabla {u'}^{n}_{h} \|^{2} \| \nabla \psi^{n+1}_{h} \|^{2} + \frac{1}{8} \| \psi^{n+1}_{h}-2\psi^{n}_{h}+\psi^{n-1}_{h} \|^{2}.\notag
\end{align}
Use the Cauchy-Schwarz-Young inequality on the first term.  Apply Lemma \ref{l1}, interpolant estimates, and Taylor's theorem on the remaining.  Then,
\begin{align}
b^{\ast}(u^{n+1}-2u^{n}+u^{n-1},\tau,\psi^{n+1}_{h}) &\leq C_{3} \| \nabla (u^{n+1}-2u^{n}+u^{n-1}) \| \|\nabla \tau\| \| \nabla \psi^{n+1}_{h} \| 
\\ &\leq \frac{CC_{r}C_{3}^2 C_{tr}^{2} |\Gamma_{N}| \Delta t^{3}}{\epsilon_{10}} \| \nabla u_{tt} \|^{2}_{L^{2}(t^{n-1},t^{n+1};L^{2}(\Omega))} + \frac{\epsilon_{10}}{r} \|\nabla \psi^{n+1}_{h} \|, \notag
\\ b^{\ast}(2u^{n}-u^{n-1},\tau - I_{h}\tau,\psi^{n+1}_{h}) &\leq C_{3}\| (2u^{n}-u^{n-1}) \| \|\nabla \tau - I_{h}\tau \| \| \nabla \psi^{n+1}_{h} \|
\\ &\leq \frac{C_{r} C_{3}^{2}}{\epsilon_{11}} C_{3}\| (2u^{n}-u^{n-1}) \|^{2} \|\nabla \tau - I_{h}\tau \|^{2} + \frac{\epsilon_{11}}{r} \| \nabla \psi^{n+1}_{h} \|^{2}, \notag
\\ b^{\ast}(2\eta^{n}-\eta^{n-1},I_{h}\tau,\psi^{n+1}_{h}) &\leq C_{4}\|\nabla I_{h}\tau \| \| \nabla \psi^{n+1}_{h} \|\Big\{2\sqrt{\| \eta^{n}\| \| \nabla \eta^{n}\|} + \sqrt{\| \eta^{n-1}\| \| \nabla \eta^{n-1}\|}\}  
\\ &\leq \frac{8C_{r} C_{4}^{2}C_{I}^{2}C_{tr}^{2}|\Gamma_{N}|}{\epsilon_{12}} \Big\{\| \eta^{n}\| \| \nabla \eta^{n}\| + \| \eta^{n-1}\| \| \nabla \eta^{n-1}\|\Big\} \notag
\\ &+ \frac{\epsilon_{12}}{r} \| \nabla \psi^{n+1}_{h} \|^{2}, \notag
\\ (\nabla (\tau - I_{h}\tau),\nabla \psi^{n+1}_{h}) &\leq \frac{C_{r}}{\epsilon_{14}} \|\nabla (\tau - I_{h}\tau)\|^{2} + \frac{\epsilon_{14}}{r} \| \nabla \psi^{n+1}_{h} \|^{2}.
\end{align}
The Cauchy-Schwarz-Young inequality, Poincar\'{e}-Friedrichs inequality and Taylor's theorem yield
\begin{align}
\frac{1}{2\Delta t} (3\zeta^{n+1} - 4\zeta^{n} + \zeta^{n-1}, \psi^{n+1}_{h}) \leq \frac{C C^{2}_{PF,2} C_{r}}{\Delta t \epsilon_1} \| \zeta_{t} \|^{2}_{L^{2}(t^{n-1},t^{n+1};L^{2}(\Omega))} + \frac{\epsilon_1}{r} \| \nabla \psi^{n+1}_{h} \|^{2}.
\end{align}
Lastly, use the Cauchy-Schwarz-Young inequality,
\begin{align}
(\nabla \zeta^{n+1},\nabla \psi^{n+1}_{h}) \leq \frac{C_{r}}{\epsilon_2} \| \nabla \zeta^{n+1} \|^{2} + \frac{\epsilon_2}{r} \| \nabla \psi^{n+1} \|^{2}.
\end{align}
Similar estimates follow for the r.h.s. terms in (\ref{feu1}), however, we must treat an additional pressure term and error term associated with the temperature,
\begin{align}
-(p^{n+1}-q^{n+1}_{h},\nabla \cdot \phi^{n+1}_{h}) &\leq \sqrt{d} \| p^{n+1}-q^{n+1}_{h} \| \| \nabla \phi^{n+1}_{h} \| \leq \frac{d C_{r}}{\epsilon_{17}} \| p^{n+1}-q^{n+1}_{h} \|^{2}
\\ &+ \frac{\epsilon_{17}}{r}\| \nabla \phi^{n+1}_{h} \|^{2},
\\ - PrRa(\xi (\theta^{n+1}-2\theta^{n}+\theta^{n-1}),\phi^{n+1}_{h}) &\leq \frac{C Pr^{2} Ra^{2} C^{2}_{PF,1} C_{r}\Delta t^{3} }{\epsilon_{18}} \| \theta_{tt} \|^{2}_{L^{2}(t^{n-1},t^{n+1};L^{2}(\Omega))}
\\ &+ \frac{\epsilon_{18}}{r} \| \nabla \phi^{n+1}_{h} \|^{2},
\\ - PrRa(\xi (2\zeta^{n} - \zeta^{n-1}),\phi^{n+1}_{h}) &\leq \frac{Pr^{2} Ra^{2} C^{2}_{PF,1} C_{r}}{\epsilon_{19}} \Big( 4\| \zeta^{n} \|^{2} + \| \zeta^{n-1} \|^{2} \Big) + \frac{\epsilon_{19}}{r} \| \nabla \phi^{n+1}_{h} \|^{2},
\\ PrRa(\xi (2\psi^{n}_{h}-\psi^{n-1}_{h}),\phi^{n+1}_{h}) &\leq \frac{Pr^{2} Ra^{2} C^{2}_{PF,1} C^{2}_{PF,2} C_{r}}{\epsilon_{20}}\| 2\psi^{n}_{h} - \psi^{n-1}_{h} \|^{2} + \frac{\epsilon_{20}}{r} \| \nabla \phi^{n+1}_{h} \|^{2},
\\ - PrRa(\xi (\tau - I_{h}\tau),\phi^{n+1}_{h}) &\leq \frac{Pr^{2} Ra^{2} C^{2}_{PF,1} C_{r}}{\epsilon_{21}}\| \tau - I_{h}\tau \|^{2} + \frac{\epsilon_{21}}{r} \| \nabla \phi^{n+1}_{h} \|^{2}.
\end{align}
Multiply equations (\ref{fet1}) and (\ref{feu1}) by $\Delta t$.  Apply the above estimates and Lemma \ref{consistency}.  Then,
\begin{multline}\label{error:thick:paramT}
\frac{1}{4} \Big\{\|\psi^{n+1}_{h}\|^{2} + \|2\psi^{n+1}_{h}-\psi^{n}_{h}\|^{2}\Big\} - \frac{1}{4} \Big\{\|\psi^{n}_{h}\|^{2} + \|2\psi^{n}_{h} - \psi^{n-1}_{h}\|^{2}\Big\} + \frac{1}{4} \|\psi^{n+1}_{h} - 2\psi^{n}_{h} + \psi^{n-1}_{h}\|^{2}
\\ + \Delta t \|\nabla \psi^{n+1}_{h}\|^{2} \leq \frac{C C^{2}_{PF,2} C_{r}}{\epsilon_1} \| \zeta_{t} \|^{2}_{L^{2}(t^{n-1},t^{n+1};L^{2}(\Omega))} + \frac{\epsilon_1 \Delta t}{r} \| \nabla \psi^{n+1}_{h} \|^{2} + \frac{C_{r}\Delta t}{\epsilon_2} \| \nabla \zeta^{n+1} \|^{2} + \frac{\epsilon_2\Delta t}{r} \| \nabla \psi^{n+1} \|^{2}
\\ + \frac{C_{r} C_{6}^{2}\Delta t}{\epsilon_3} \| \nabla u^{n+1} \|^{2} \| \| \zeta^{n+1} \| \| \nabla \zeta^{n+1} \| + \frac{\epsilon_3\Delta t}{r} \| \nabla \psi^{n+1}_{h} \|^{2} + \frac{C C_{r} C_{3}^{2} \Delta t^{4}}{\epsilon_4} \| \nabla \theta^{n+1}_{h} \|^{2} \| \nabla u_{tt} \|^{2}_{L^{2}(t^{n-1},t^{n+1};L^{2}(\Omega))}
\\ + \frac{\epsilon_4\Delta t}{r} \| \nabla \psi^{n+1}_{h} \|^{2} + \frac{8C_{r} C_{4}^{2}\Delta t}{\epsilon_5} \| \nabla \theta^{n+1}_{h} \|^{2} \Big\{\| \eta^{n}\| \| \nabla \eta^{n}\| + \| \eta^{n-1}\| \| \nabla \eta^{n-1}\|\Big\} + \frac{\epsilon_5\Delta t}{r} \| \nabla \psi^{n+1}_{h} \|^{2}
\\ + \epsilon_{6}\Delta t\| \nabla \psi^{n+1}_{h} \|^2 + \frac{C_{4}^{2} C_{\theta}^{2}\Delta t}{8\epsilon_{6}\delta_{6}} \| 2\phi^{n}_{h}-\phi^{n-1}_{h} \|^{2} + \frac{C_{4}^{2} C_{\theta}^{2}\delta_{6}\Delta t}{2\epsilon_{6}} \Big( \| \nabla \phi^{n}_{h} \|^{2} + \| \nabla \phi^{n-1}_{h} \|^{2}\Big)
\\ + \frac{C C_{r} C_{3}^{2} \Delta t^{4}}{\epsilon_7} \| \nabla {u'}^{n}_{h} \|^{2}\| \nabla \zeta_{tt} \|^{2}_{L^{2}(t^{n-1},t^{n+1};L^{2}(\Omega))} + \frac{\epsilon_7\Delta t}{r} \| \nabla \psi^{n+1}_{h} \|^{2} + \frac{2C_{6}^{2} C_{inv,2} \Delta t^2}{h} \| \nabla {u'}^{n}_{h} \|^{2} \| \nabla \psi^{n+1}_{h} \|^{2}
\\ + \frac{1}{8} \| \psi^{n+1}_{h}-2\psi^{n}_{h}+\psi^{n-1}_{h} \|^{2} + \frac{C C_{r} C_{3}^{2}\Delta t^{4}}{\epsilon_{9}} \| \nabla {u'}^{n}_{h} \|^{2} \| \nabla \theta_{tt} \|^{2}_{L^{2}(t^{n-1},t^{n+1};L^{2}(\Omega))} + \frac{\epsilon_{9}\Delta t}{r} \| \nabla \psi^{n+1}_{h}\|^{2}
\\ + \frac{CC_{r}C_{3}^2 C_{tr}^{2} |\Gamma_{N}| \Delta t^{4}}{\epsilon_{10}} \| \nabla u_{tt} \|^{2}_{L^{2}(t^{n-1},t^{n+1};L^{2}(\Omega))} + \frac{\epsilon_{10}\Delta t}{r} \|\nabla \psi^{n+1}_{h} \| + \frac{C_{r} C_{3}^{2}\Delta t}{\epsilon_{11}}\| (2u^{n}-u^{n-1}) \|^{2} \|\nabla \tau - I_{h}\tau \|^{2}
\\ + \frac{\epsilon_{11}\Delta t}{r} \| \nabla \psi^{n+1}_{h} \|^{2} + \frac{8C_{r} C_{4}^{2}C_{I}^{2}C_{tr}^{2}|\Gamma_{N}|\Delta t}{\epsilon_{12}} \Big\{\| \eta^{n}\| \| \nabla \eta^{n}\| + \| \eta^{n-1}\| \| \nabla \eta^{n-1}\|\Big\} + \frac{\epsilon_{12}\Delta t}{r} \| \nabla \psi^{n+1}_{h} \|^{2}
\\ + \epsilon_{13}\Delta t \| \nabla \psi^{n+1}_{h} \|^2 + \frac{C_{4}^{2} C_{I}^{2}C_{tr}^{2} |\Gamma_{N}|\Delta t}{8\epsilon_{13}\delta_{13}} \| 2\phi^{n}_{h}-\phi^{n-1}_{h} \|^{2}  + \frac{C_{4}^{2} C_{I}^{2}C_{tr}^{2} |\Gamma_{N}|\delta_{13}\Delta t}{2\epsilon_{13}} \Big( \| \nabla \phi^{n}_{h} \|^{2} + \| \nabla \phi^{n-1}_{h} \|^{2}\Big)
\\ + \frac{C_{r}\Delta t}{\epsilon_{14}} \|\nabla (\tau - I_{h}\tau)\|^{2} + \frac{\epsilon_{14}\Delta t}{r} \| \nabla \psi^{n+1}_{h} \|^{2} + \frac{C C_{PF,2}^{2} C_{r}\Delta t^{4}}{\epsilon_{29}} \|\theta_{ttt}\|^{2}_{L^{2}(t^{n-1},t^{n+1};L^{2}(\Omega))}  + \frac{\epsilon_{29}\Delta t}{r} \| \nabla \psi^{n+1}_{h} \|^{2}
\end{multline}
\noindent and
\begin{multline}\label{error:thick:paramU}
\frac{1}{4} \Big\{\|\phi^{n+1}_{h}\|^{2} + \|2\phi^{n+1}_{h}-\phi^{n}_{h}\|^{2}\Big\} - \frac{1}{4} \Big\{\|\phi^{n}_{h}\|^{2} + \|2\phi^{n}_{h} - \phi^{n-1}_{h}\|^{2}\Big\} + \frac{1}{4} \|\phi^{n+1}_{h} - 2\phi^{n}_{h} + \phi^{n-1}_{h}\|^{2}
\\ + Pr \Delta t \|\nabla \phi^{n+1}_{h}\|^{2} \leq \frac{C C_{r}C^{2}_{PF,1}}{\epsilon_{15}} \| \eta_{t} \|^{2}_{L^{2}(t^{n-1},t^{n+1};L^{2}(\Omega))} + \frac{\Delta t \epsilon_{15}}{r} \| \nabla \phi^{n+1}_{h} \|^{2} + \frac{C_{r} Pr^{2} \Delta t}{\epsilon_{16}}\| \nabla \eta^{n+1} \|^{2}
\\ + \frac{\Delta t \epsilon_{16}}{r} \| \nabla \phi^{n+1}_{h} \|^{2} + \frac{dC_{r}\Delta t}{\epsilon_{17}}\| p^{n+1} - q^{n+1}_{h} \|^{2} + \frac{\Delta t \epsilon_{17}}{r} \| \nabla \phi^{n+1}_{h} \|^{2} + \frac{C Pr^{2} Ra^{2} C^{2}_{PF,1} C_{r}\Delta t^{4} }{\epsilon_{18}} \| \theta_{tt} \|^{2}_{L^{2}(t^{n-1},t^{n+1};L^{2}(\Omega))}
\\ + \frac{\epsilon_{18}\Delta t}{r} \| \nabla \phi^{n+1}_{h} \|^{2} + \frac{Pr^{2} Ra^{2} C^{2}_{PF,1} C_{r}\Delta t}{\epsilon_{19}} \Big( 4\| \zeta^{n} \|^{2} + \| \zeta^{n-1} \|^{2} \Big) + \frac{\epsilon_{19}\Delta t}{r} \| \nabla \phi^{n+1}_{h} \|^{2}
\\ + \frac{Pr^{2} Ra^{2} C^{2}_{PF,1} C^{2}_{PF,2} C_{r}\Delta t}{\epsilon_{20}}\| 2\psi^{n}_{h} - \psi^{n-1}_{h} \|^{2} + \frac{\epsilon_{20}\Delta t}{r} \| \nabla \phi^{n+1}_{h} \|^{2} + \frac{Pr^{2} Ra^{2} C^{2}_{PF,1} C_{r}\Delta t}{\epsilon_{21}}\| \tau - I_{h}\tau \|^{2}
\\ + \frac{\epsilon_{21}\Delta t}{r} \| \nabla \phi^{n+1}_{h} \|^{2} + \frac{C_{5} C_{r} \Delta t}{\epsilon_{22}}\| \nabla u^{n+1} \|^{2} \| \eta^{n+1} \| \| \nabla \eta^{n+1} \| + \frac{\Delta t \epsilon_{22}}{r} \| \nabla \phi^{n+1}_{h} \|^{2} 
\\ + \frac{C C_{r}C_{1}^{2} \Delta t^{4}}{\epsilon_{23}} \| \nabla u^{n+1}_{h} \|^{2} \| \nabla u_{tt} \|^{2}_{L^{2}(t^{n-1},t^{n+1};L^{2}(\Omega))} + \frac{\Delta t \epsilon_{23}}{r} \| \nabla \phi^{n+1}_{h} \|^{2}
\\ + \frac{8C_{r} C_{2}^{2}\Delta t}{\epsilon_{24}} \| \nabla u^{n+1}_{h} \|^{2} \Big\{\| \eta^{n}\| \| \nabla \eta^{n}\| + \| \eta^{n-1}\| \| \nabla \eta^{n-1}\|\Big\} + \frac{\Delta t \epsilon_{24}}{r} \| \nabla \phi^{n+1}_{h} \|^{2} + \epsilon_{25} \Delta t \| \nabla \phi^{n+1}_{h} \|^{2}
\\ + \frac{C_{2}^{2} C_{u}^{2} \Delta t \delta_{25}}{8\epsilon_{25}}\| 2\phi^{n}_{h} - \phi^{n-1}_{h} \|^{2} + \frac{C_{2}^{2} C_{u}^{2} \Delta t}{2\epsilon_{25} \delta_{25}} \Big( \| \nabla \phi^{n}_{h} \|^{2} + \| \nabla \phi^{n-1}_{h} \|^{2} \Big)
\\ + \frac{C C_{r} C_{1}^{2} \Delta t^{4}}{\epsilon_{26}} \| \nabla {u'}^{n}_{h} \|^{2}\| \nabla \eta_{tt} \|^{2}_{L^{2}(t^{n-1},t^{n+1};L^{2}(\Omega))} + \frac{\Delta t \epsilon_{26}}{r} \| \nabla \phi^{n+1}_{h} \|^{2}
\\ + \frac{2 C_{5}^{2} C_{inv,1} \Delta t^{2}}{h} \| \nabla {u'}^{n}_{h} \|^{2} \| \nabla \phi^{n+1}_{h} \|^{2} + \frac{1}{8} \| \phi^{n+1}_{h} - 2\phi^{n}_{h} + \phi^{n-1}_{h} \|^{2}
\\ + \frac{C C_{r}C_{1} \Delta t^{4}}{\epsilon_{28}}\| \nabla {u'}^{n}_{h} \|^{2} \| \nabla u_{tt} \|^{2}_{L^{2}(t^{n-1},t^{n+1};L^{2}(\Omega))} + \frac{\Delta t \epsilon_{28}}{r} \| \nabla \phi^{n+1}_{h} \|^{2}
\\ + \frac{C C^{2}_{PF,1}C_{r} \Delta t^{4}}{\epsilon_{30}} \| u_{ttt} \|^{2}_{L^{2}(t^{n-1},t^{n+1};L^{2}(\Omega))} + \frac{\Delta \epsilon_{30}}{r} \| \nabla \phi^{n+1}_{h} \|^{2}.
\end{multline}
Combine (\ref{error:thick:paramT}) and (\ref{error:thick:paramU}), choose free parameters appropriately, reorganize, use condition (\ref{c1}) and Theorem \ref{t1}.  Add $\|\psi^{n}_{h}\|^{2}$ and $\|\phi^{n}_{h}\|^{2}$ to the r.h.s. and take the maximum over all constants on the r.h.s.  Then,
\begin{multline} 
\frac{1}{4} \Big\{\|\psi^{n+1}_{h}\|^{2} + \|2\psi^{n+1}_{h}-\psi^{n}_{h}\|^{2}\Big\} - \frac{1}{4} \Big\{\|\psi^{n}_{h}\|^{2} + \|2\psi^{n}_{h} - \psi^{n-1}_{h}\|^{2}\Big\} + \frac{1}{8} \|\psi^{n+1}_{h} - 2\psi^{n}_{h} + \psi^{n-1}_{h}\|^{2}
\\ + \frac{\Delta t}{4} \|\nabla \psi^{n+1}_{h}\|^{2} + \frac{1}{4} \Big\{\|\phi^{n+1}_{h}\|^{2} + \|2\phi^{n+1}_{h}-\phi^{n}_{h}\|^{2}\Big\} - \frac{1}{4} \Big\{\|\phi^{n}_{h}\|^{2} + \|2\phi^{n}_{h} - \phi^{n-1}_{h}\|^{2}\Big\}
\\ + \frac{1}{8} \|\phi^{n+1}_{h} - 2\phi^{n}_{h} + \phi^{n-1}_{h}\|^{2} + \frac{Pr \Delta t}{4} \|\nabla \phi^{n+1}_{h}\|^{2} + \frac{Pr \Delta t}{4} \Big( \|\nabla \phi^{n+1}_{h}\|^{2} - \|\nabla \phi^{n}_{h}\|^{2}\Big) + \frac{Pr \Delta t}{8} \Big( \|\nabla \phi^{n}_{h}\|^{2} - \|\nabla \phi^{n-1}_{h}\|^{2}\Big)
\\ \leq C\Big\{\Delta t \| \zeta_{t} \|^{2}_{L^{2}(t^{n-1},t^{n+1};L^{2}(\Omega))} + \Delta t \| \nabla \zeta^{n+1} \|^{2} + \Delta t \| \zeta^{n+1} \| \| \nabla \zeta^{n+1} \| + \Delta t^{4}
\\ + \Delta t \Big\{\| \eta^{n}\| \| \nabla \eta^{n}\| + \| \eta^{n-1}\| \| \nabla \eta^{n-1}\|\Big\} + \Delta t \Big(\|\phi^{n}_{h}\|^{2} + \| 2\phi^{n}_{h} - \phi^{n-1}_{h} \|^{2}\Big) + h \Delta t^{3} + \Delta t \|\nabla \tau - I_{h}\tau \|^{2}
\\ + \Delta t \| \eta_{t} \|^{2}_{L^{2}(t^{n-1},t^{n+1};L^{2}(\Omega))} + \Delta t \| \nabla \eta^{n+1} \|^{2} + \Delta t \| p^{n+1} - q^{n+1}_{h} \|^{2} + \Delta t \Big(4\| \zeta^{n}\|^{2} + \| \nabla \zeta^{n-1}\|^{2}\Big)
\\ + \Delta t \Big(\|\psi^{n}_{h}\|^{2} + \| 2\psi^{n}_{h} - \psi^{n-1}_{h} \|^{2}\Big) + \Delta t \| \tau - I_{h}\tau\|^{2} + \Delta t \|\eta^{n+1} \| \| \nabla \eta^{n+1} \|\Big\}.
\end{multline}
Multiply by 4, sum from $n = 1$ to $n = N-1$, apply Lemma \ref{l4}, take infimums over $X_{h}$, $Q_{h}$, and $\hat{W_{h}}$, and renorm. Then,
\begin{multline*}
\|\psi^{N}_{h}\|^{2} + \|2\psi^{N}_{h}-\psi^{N-1}_{h}\|^{2} + \|\phi^{N}_{h}\|^{2} + \|2\phi^{N}_{h} - \phi^{N-1}_{h}\|^{2} + \frac{1}{2}\sum_{n = 1}^{N-1}\big(\|\psi^{n+1}_{h} - 2\psi^{n}_{h} + \psi^{n-1}_{h}\|^{2} + \|\phi^{n+1}_{h} - 2\phi^{n}_{h} + \phi^{n-1}_{h}\|^{2}\big)
\\ + \Delta t \sum_{n = 1}^{N-1} \|\nabla \psi^{n+1}_{h}\|^{2} + Pr \Delta t \sum_{n = 1}^{N-1} \|\nabla \phi^{n+1}_{h}\|^{2} + \frac{Pr \Delta t}{2} \Big(\|\nabla \phi^{N}_{h}\|^{2} +\frac{1}{2} \|\nabla \phi^{N-1}_{h}\|^{2}\Big)
\\ \leq \exp(Ct^{\ast}) \Big\{ \Delta t \inf_{S_{h} \in W_{\Gamma_{1,h}}} \Big( \vertiii{\theta - S_{h}}^{2}_{\infty,0} + \vertiii{\theta - S_{h}}_{\infty,0} \vertiii{\nabla (\theta - S_{h})}_{\infty,0} + \vertiii{\nabla (\theta - S_{h})}^{2}_{\infty,0} + \vertiii{(\theta - S_{h})_{t}}^{2}_{\infty,0}
\\ + h \Delta t^{2} \vertiii{(\theta - S_{h})_{tt}}^{2}_{\infty,0} \Big) + \Delta t \inf_{v_{h} \in X_{h}} \Big( \vertiii{u - v_{h} }^{2}_{\infty,0} + \vertiii{u - v_{h} }^{2}_{\infty,0} \vertiii{\nabla (u - v_{h}) }_{\infty,0} + \vertiii{\nabla (u - v_{h}) }^{2}_{\infty,0}
\\ + \vertiii{(u - v_{h})_{t}}^{2}_{\infty,0} + h \Delta t^{2} \vertiii{(u - v_{h})_{tt}}^{2}_{\infty,0} \Big) + \Delta t \inf_{q_{h} \in Q_{h}} \vertiii{ p - q_{h} }^{2}_{\infty,0}
\\ + \Delta t \inf_{S_{h} \in W_{h}} \Big( \vertiii{\tau - I_{h}\tau}^{2}_{\infty,0} + \vertiii{\nabla (\tau - I_{h}\tau) }^{2}_{\infty,0}\Big) + h\Delta t^{3} + \Delta t^{4} \Big\}
\\ + \|\psi^{0}_{h}\|^{2} + \|2\psi^{1}_{h}-\psi^{0}_{h}\|^{2} + \|\phi^{0}_{h}\|^{2} + \|2\phi^{1}_{h} - \phi^{0}_{h}\|^{2} + \frac{Pr \Delta t}{2} \Big(\|\nabla \phi^{1}_{h}\|^{2} +\frac{1}{2} \|\nabla \phi^{0}_{h}\|^{2}\Big).
\end{multline*}
The result follows by the relationship $e^{n}_{T} = e^{n}_{\theta} + \tau - I_{h}\tau$ and the triangle inequality.
\end{proof}
\begin{corollary}
	Suppose the assumptions of Theorem \ref{t1} hold with $k=m=2$.  Further suppose that the finite element spaces ($X_{h}$,$Q_{h}$,$W_{h}$) are given by P2-P1-P2 (Taylor-Hood), then the errors in velocity and temperature satisfy
	\begin{multline*}
\frac{1}{2}\|e^{N}_{T}\|^{2} + \frac{1}{2}\|2e^{N}_{T}-e^{N-1}_{T}\|^{2} + \|e^{N}_{u}\|^{2} + \|2e^{N}_{u} - e^{N-1}_{u}\|^{2} + \frac{1}{2}\sum_{n = 1}^{N-1}\big(\|e^{n+1}_{T} - 2e^{n}_{T} + e^{n-1}_{T}\|^{2} + \|e^{n+1}_{u} - 2e^{n}_{u} + e^{n-1}_{u}\|^{2}\big)
\\ + \frac{\Delta t}{2} \sum_{n = 1}^{N-1} \|\nabla e^{n+1}_{T}\|^{2} + Pr \Delta t \sum_{n = 1}^{N-1} \|\nabla e^{n+1}_{u}\|^{2} + \frac{Pr \Delta t}{2} \Big(\|\nabla e^{N}_{u}\|^{2} +\frac{1}{2} \|\nabla e^{N-1}_{u}\|^{2}\Big)
\\ \leq C \Big\{ h^{6}\Delta t + h^{5}\Delta t + h^{4}\Delta t + h^{7}\Delta t^{2} + h\Delta t^{3} + h^{4}\Delta t + h \Delta t^{3} + \Delta t^{4}
\\ + \|\zeta^{0}_{h}\|^{2} + \|2\zeta^{1}_{h}-\zeta^{0}_{h}\|^{2} + \|\eta^{0}_{h}\|^{2} + \|2\eta^{1}_{h} - \phi^{0}_{h}\|^{2} + \frac{Pr \Delta t}{2} \Big(\|\nabla \eta^{1}_{h}\|^{2} +\frac{1}{2} \|\nabla \eta^{0}_{h}\|^{2}\Big)
\\ + \frac{1}{2}\Big(\|e^{0}_{T}\|^{2} + \|2e^{1}_{T}-e^{0}_{T}\|^{2}\Big) + \|e^{0}_{u}\|^{2} + \|2e^{1}_{u} - e^{0}_{u}\|^{2} + \frac{Pr \Delta t}{2} \Big(\|\nabla e^{1}_{u}\|^{2} +\frac{1}{2} \|\nabla e^{0}_{u}\|^{2}\Big)\Big\}.
	\end{multline*}
\end{corollary}
\begin{corollary}
	Suppose the assumptions of Theorem \ref{t1} hold with $k=m=1$.  Further suppose that the finite element spaces ($X_{h}$,$Q_{h}$,$W_{h}$) are given by P1b-P1-P1b (MINI element), then the errors in velocity and temperature satisfy
	\begin{multline*}
\frac{1}{2}\|e^{N}_{T}\|^{2} + \frac{1}{2}\|2e^{N}_{T}-e^{N-1}_{T}\|^{2} + \|e^{N}_{u}\|^{2} + \|2e^{N}_{u} - e^{N-1}_{u}\|^{2} + \frac{1}{2}\sum_{n = 1}^{N-1}\big(\|e^{n+1}_{T} - 2e^{n}_{T} + e^{n-1}_{T}\|^{2} + \|e^{n+1}_{u} - 2e^{n}_{u} + e^{n-1}_{u}\|^{2}\big)
\\ + \frac{\Delta t}{2} \sum_{n = 1}^{N-1} \|\nabla e^{n+1}_{T}\|^{2} + Pr \Delta t \sum_{n = 1}^{N-1} \|\nabla e^{n+1}_{u}\|^{2} + \frac{Pr \Delta t}{2} \Big(\|\nabla e^{N}_{u}\|^{2} +\frac{1}{2} \|\nabla e^{N-1}_{u}\|^{2}\Big)
\\ \leq C \Big\{ h^{4}\Delta t + h^{3}\Delta t + h^{2}\Delta t + h^{5}\Delta t^{2} + h\Delta t^{3} + h^{2} \Delta t + h \Delta t^{3} + \Delta t^{4}
\\ + \|\zeta^{0}_{h}\|^{2} + \|2\zeta^{1}_{h}-\zeta^{0}_{h}\|^{2} + \|\eta^{0}_{h}\|^{2} + \|2\eta^{1}_{h} - \phi^{0}_{h}\|^{2} + \frac{Pr \Delta t}{2} \Big(\|\nabla \eta^{1}_{h}\|^{2} +\frac{1}{2} \|\nabla \eta^{0}_{h}\|^{2}\Big)
\\ + \frac{1}{2}\Big(\|e^{0}_{T}\|^{2} + \|2e^{1}_{T}-e^{0}_{T}\|^{2}\Big) + \|e^{0}_{u}\|^{2} + \|2e^{1}_{u} - e^{0}_{u}\|^{2} + \frac{Pr \Delta t}{2} \Big(\|\nabla e^{1}_{u}\|^{2} +\frac{1}{2} \|\nabla e^{0}_{u}\|^{2}\Big)\Big\}.
	\end{multline*}
\end{corollary}
\begin{figure}
	\includegraphics[width=\textwidth,height=\textheight,keepaspectratio]{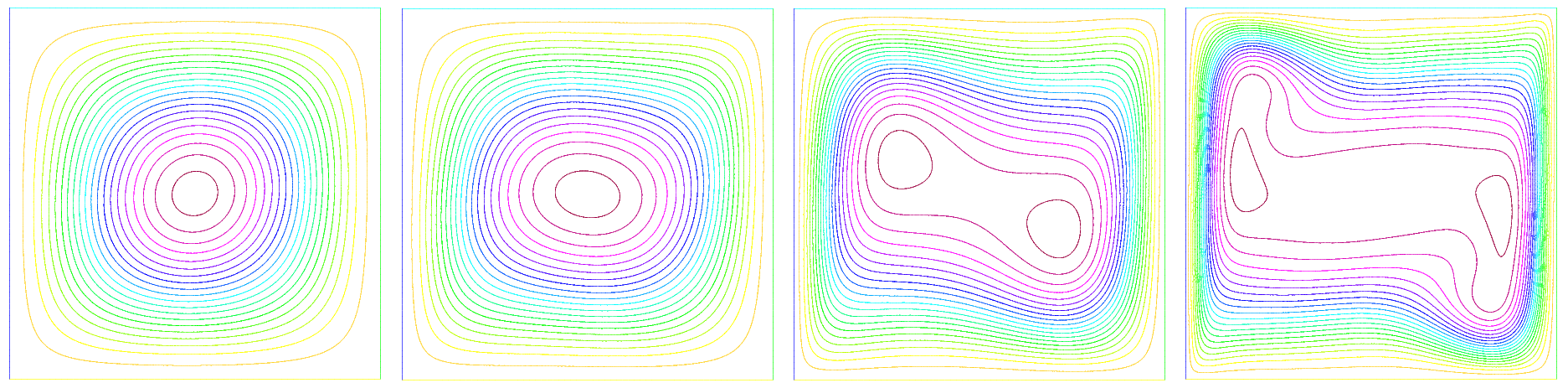}
	\caption{Streamlines for $Ra = 10^3, 10^4, 10^5,$ and $10^6$, from left to right, respectively.}
\end{figure}
\begin{figure}
	\includegraphics[width=\textwidth,height=\textheight,keepaspectratio]{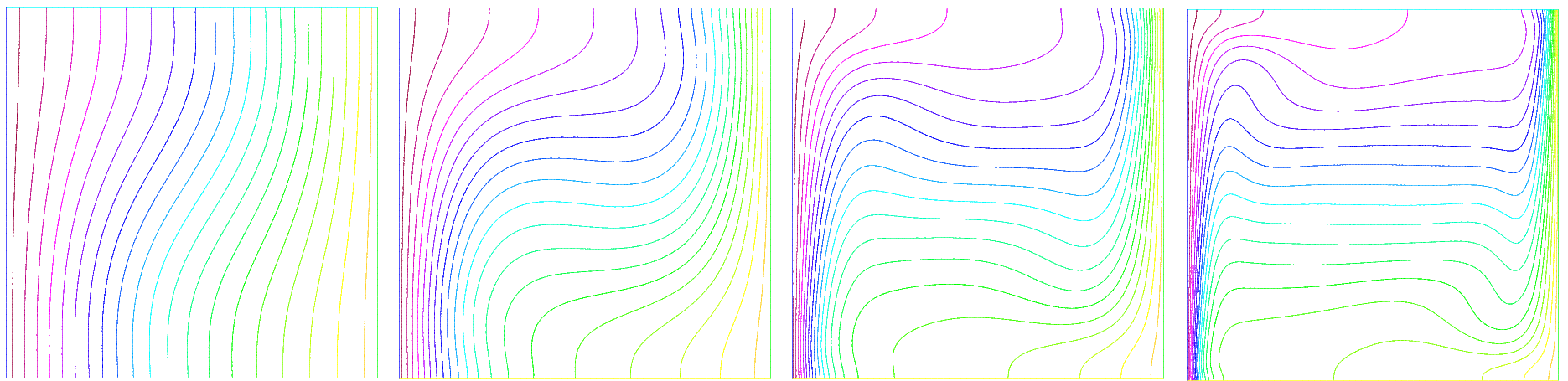}
	\caption{Isotherms for $Ra = 10^3, 10^4, 10^5,$ and $10^6$, from left to right, respectively.}
\end{figure}
\section{Numerical Experiments}
In this section, we illustrate the stability and convergence of the numerical scheme described by (\ref{scheme:one:velocity}) - (\ref{scheme:one:temperature}) using Taylor-Hood (P2-P1-P2) elements to approximate the average velocity, pressure, and temperature.  The numerical experiments include the double pane window benchmark problem of De Vahl Davis \cite{Davis} and a convergence experiment with an analytical solution devised through the method of manufactured solutions.  The software used for all tests is \textsc{FreeFem}$++$ \cite{Hecht}.
\subsection{Stability condition}
The constant appearing in condition \ref{c1} is estimated by pre-computations for the double pane window problem appearing below.  We set $C_{\dagger} = 1$.  The first condition is used and checked at each iteration.  If violated, the timestep is halved and the timestep is repeated.  The timestep is never increased.  The condition is violated two times during the computation of the double pane window problem with $Ra = 10^{6}$ in Section 5.3.
\subsection{Perturbation generation}
The bred vector (BV) algorithm of Toth and Kalnay \cite{Toth} is used to generate perturbations in the double pane window problem.  The BV algorithm produces a perturbation with maximal separation rate.  We set $J = 2$ in all experiments.  An initial random positive/negative perturbation pair was generated $\pm \epsilon = \pm (\epsilon_{1},\epsilon_{2},\epsilon_{3})$ with $\epsilon_{i} \in (0,0.01) \; \forall i = 1,2,3$.  Denote the control and perturbed numerical approximations $\chi^{n}_{h}$ and $\chi^{n}_{p,h}$, respectively.  Then, a bred vector $bv(\chi;\epsilon_{i})$ is generated via:

\textbf{Step one:}  Given $\chi^{0}_{h}$ and $\epsilon_{i}$, put $\chi^{0}_{p,h} = \chi^{0}_{h} + \epsilon_{i}$.  Select time reinitialization interval $\delta t \geq \Delta t$ and let \indent $t^{k} = k \delta t$ with $0 \leq k \leq k^{\ast} \leq N$.

\indent \textbf{Step two:} Compute $\chi^{k}_{h}$ and $\chi^{k}_{p,h}$.  Calculate $bv(\chi^{k};\epsilon_{i}) = \frac{\epsilon_{i}}{\| \chi^{k}_{p,h} - \chi^{k}_{h} \|} (\chi^{k}_{p,h} - \chi^{k}_{h})$.

\indent \textbf{Step three:}  Put $\chi^{k}_{p,h} = \chi^{k}_{h} + bv(\chi^{k};\epsilon_{i}) $.

\indent \textbf{Step four:} Repeat \textbf{Step two} with $k = k + 1$.

\indent \textbf{Step five:} Put $bv(\chi;\epsilon_{i}) = bv(\chi^{k^{\ast}};\epsilon_{i})$.

\noindent A positive/negative perturbed initial condition pair is generated via $\chi_{\pm} = \chi^{0} + bv(\chi;\pm \epsilon_{i})$. Moreover, we let $\delta t = \Delta t = 0.001$ and $k^{\ast} = 5$.

\subsection{The double pane window problem}
The first numerical experiment is the benchmark problem of De Vahl Davis \cite{Davis}.  The problem is the two-dimensional flow of a fluid in an unit square cavity with $Pr = 0.71$. Both velocity components are zero on the boundaries. The horizontal walls are insulated and the left and right vertical walls are maintained at temperatures $T(0,y,t) = 1$ and $T(1,y,t) = 0$, respectively; recall Figure 1.  We let $10^3 \leq Ra \leq 10^6$.  The initial conditions for velocity and temperature are generated via the BV algorithm in Section 5.2,
\begin{align*}
u_{\pm}(x,y,0) := u(x,y,0;\omega_{1,2}) &= (1 + bv(u;\pm \epsilon_{1}), 1 + bv(u;\pm \epsilon_{2}))^{T}, \\
T_{\pm}(x,y,0) := T(x,y,0;\omega_{1,2}) &= 1 + bv(T;\pm \epsilon_{3}).
\end{align*}
Both $f(x,t;\omega_{j})$ and $g(x,t;\omega_{j})$ are identically zero for $j = 1,2$.  The finite element mesh is a division of $[0,1]^{2}$ into $64^{2}$ squares with diagonals connected with a line within each square in the same direction.  The stopping condition is
\begin{equation*}
\max_{1\leq n \leq N-1}\big\{\frac{\| u^{n+1}_{h} - u^{n}_{h}\|}{\| u^{n+1}_{h}\|},\frac{\|T^{n+1}_{h} - T^{n}_{h}\|}{\| T^{n+1}_{h}\|}\big\} \leq 10^{-5}
\end{equation*}
and initial timestep $\Delta t = 0.001$.  The first iterate was computed with the trapezoidal rule for each ensemble member.  The timestep was halved twice to $0.00025$ to maintain stability for $Ra = 10^{6}$.  Several quantities are compared with benchmark solutions in the literature.  These include the maximum vertical velocity at $y = 0.5$, $\max_{x \in \Omega_{h}} {u_{2}(x,0.5,t^{\ast})}$, and maximum horizontal velocity at $x = 0.5$,  $\max_{y \in \Omega_{h}} {u_{1}(0.5,y,t^{\ast})}$.  We present our computed values for the average flow in Tables 1 and 2 alongside several of those seen in the literature.  Furthermore, the local Nusselt number is calculated at the cold (+) and hot walls (-), respectively, via
\begin{equation*}
Nu_{local} =  \pm \frac{\partial{T}}{\partial{x}}.
\end{equation*}
The average Nusselt number on the vertical boundary at x = 0 is calculated via
\begin{equation*}
	Nu_{avg} =  \int^{1}_{0} Nu_{local} dy.
\end{equation*}
Figure 2 presents the plots of $Nu_{local}$ at the hot and cold walls.  Table 3 presents computed values of $Nu_{avg}$ alongside several of those seen in the literature.  Figures 3 and 4 present the velocity streamlines and temperature isotherms for the averages.  All results appear to be in good agreement with the benchmark values in the literature \cite{Davis, Manzari, Wan, Cibik, Zhang}.
\begin{figure}
	\centering
	\includegraphics[width=5.5in,height=\textheight, keepaspectratio]{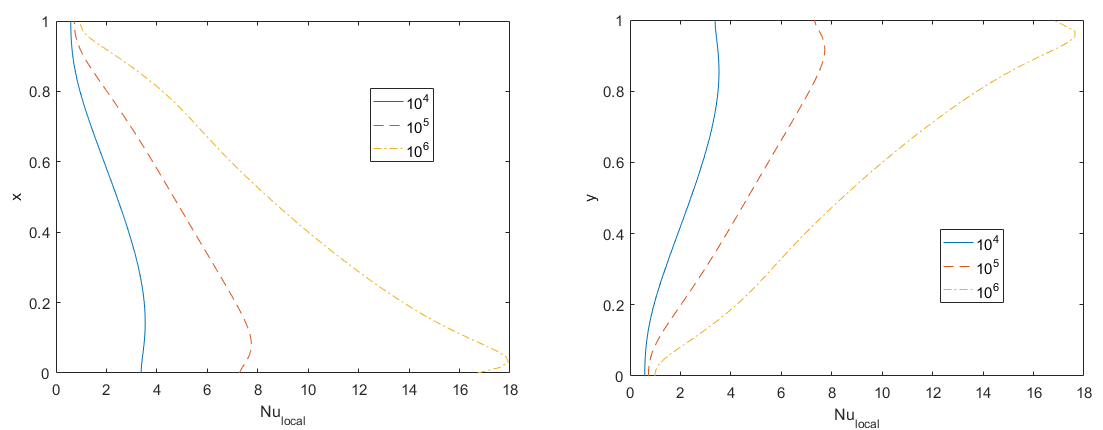}
	\caption{Variation of the local Nusselt number at the hot (left) and cold walls (right).}
\end{figure}
\begin{adjustbox}{max width=\textwidth}
	\begin{tabular}{ c  c  c  c  c  c  c  c }
		\hline			
		Ra & Present study & Ref. \cite{Davis} & Ref. \cite{Manzari} & Ref. \cite{Wan} & Ref. \cite{Cibik} & Ref. \cite{Zhang} \\
		\hline
		$10^{4}$ & 16.18 (64$\times$64) & 16.18 (41$\times$41) & 16.10 (71$\times$71) & 16.10 (101$\times$101) & 15.90 (11$\times$11) & 16.18 (64$\times$64)\\
		$10^{5}$ & 34.72 (64$\times$64) & 34.81 (81$\times$81) & 34 (71$\times$71) & 34 (101$\times$101) & 33.51 (21$\times$21) & 34.74 (64$\times$64) \\
		$10^{6}$ & 64.78 (64$\times$64) & 65.33 (81$\times$81) & 65.40 (71$\times$71) & 65.40 (101$\times$101) & 65.52 (32$\times$32) & 64.81 (64$\times$64)\\
		\hline  
	\end{tabular}
\end{adjustbox}
\captionof{table}{Comparison of maximum horizontal velocity at x = 0.5 together with mesh size used in computation for the double pane window problem.}
\begin{adjustbox}{max width=\textwidth}
	\begin{tabular}{ c  c  c  c  c  c  c  c }
		\hline			
		Ra & Present study & Ref. \cite{Davis} & Ref. \cite{Manzari} & Ref. \cite{Wan} & Ref. \cite{Cibik} & Ref. \cite{Zhang} \\
		\hline
		$10^{4}$ & 19.60 (64$\times$64) & 19.51 (41$\times$41) & 19.90 (71$\times$71) & 19.79 (101$\times$101) & 19.91 (11$\times$11) & 19.62 (64$\times$64)\\
		$10^{5}$ & 68.53 (64$\times$64) & 68.22 (81$\times$81) & 70 (71$\times$71) & 70.63 (101$\times$101) & 70.60 (21$\times$21) & 68.48 (64$\times$64) \\
		$10^{6}$ & 215.89 (64$\times$64) & 216.75 (81$\times$81) & 228 (71$\times$71) & 227.11 (101$\times$101) & 228.12 (32$\times$32) & 220.44 (64$\times$64)\\
		\hline  
	\end{tabular}
\end{adjustbox}
\captionof{table}{Comparison of maximum horizontal velocity at y = 0.5 together with mesh size used in computation for the double pane window problem.}
\begin{adjustbox}{max width=\textwidth}
	\begin{tabular}{ c  c  c  c  c  c  c  c }
		\hline			
		Ra & Present study & Ref. \cite{Davis} & Ref. \cite{Manzari} & Ref. \cite{Wan} & Ref. \cite{Cibik} & Ref. \cite{Zhang} \\
		\hline
		$10^{4}$ & 2.25 (64$\times$64) & 2.24 (41$\times$41) & 2.08 (71$\times$71) & 2.25 (101$\times$101) & 2.15 (11$\times$11) & 2.25 (64$\times$64)\\
		$10^{5}$ & 4.53 (64$\times$64) & 4.52 (81$\times$81) & 4.30 (71$\times$71) & 4.59 (101$\times$101) & 4.35 (21$\times$21) & 4.53 (64$\times$64) \\
		$10^{6}$ & 8.89 (64$\times$64) & 8.92 (81$\times$81) & 8.74 (71$\times$71) & 8.97 (101$\times$101) & 8.83 (32$\times$32) & 8.87 (64$\times$64)\\
		\hline  
	\end{tabular}
\end{adjustbox}
\captionof{table}{Comparison of average Nusselt number on the vertical boundary at x = 0 together with mesh size used in computation for the double pane window problem.}
\subsection{Numerical convergence study}
In this section, we illustrate the convergence rates for the proposed algorithm (\ref{scheme:one:velocity}) - (\ref{scheme:one:temperature}).  The unperturbed solution is given by
\begin{align*}
u(x,y,t) &= 10(1+0.1t)(10x^2(x-1)^2y(y-1)(2y-1), -10x(x-1)(2x-1)y^2(y-1)^2)^T, \\
T(x,y,t) &= u_{1}(x,y,t) + u_{2}(x,y,t) + 1-x, \\
p(x,y,t) &= 10(1+0.1t)(2x-1)(2y-1),
\end{align*}
with $Pr = 1.0$, $Ra = 100$, and $\Omega = [0,1]^{2}$.  The perturbed solutions are given by
\begin{align*}
u(x,y,t;\omega_{1,2}) = (1 + \epsilon_{1,2})u(x,y,t), \\
T(x,y,t;\omega_{1,2}) = (1 + \epsilon_{1,2})T(x,y,t), \\
p(x,y,t;\omega_{1,2}) = (1 + \epsilon_{1,2})p(x,y,t), 
\end{align*} 
where $\epsilon_{1} = 1e-2 = -\epsilon_{2}$ and both forcing and boundary terms are adjusted appropriately.  The perturbed solutions satisfy the following relations,
\begin{align*}
< u > = 0.5\big(u(x,y,t;\omega_{1}) + u(x,y,t;\omega_{2}) \big) = u(x,y,t), \\
< T > = 0.5\big(T(x,y,t;\omega_{1}) + T(x,y,t;\omega_{2}) \big) = T(x,y,t), \\
< p > = 0.5\big(p(x,y,t;\omega_{1}) + p(x,y,t;\omega_{2}) \big) = p(x,y,t).
\end{align*}
The finite element mesh $\Omega_{h}$ is a Delaunay triangulation generated from $m$ points on each side of $\Omega$.  We calculate errors in the approximations of the average velocity and temperature with the $L^{\infty}(0,t^{\ast};L^{2}(\Omega))$ and $L^{2}(0,t^{\ast};H^{1}(\Omega))$ norms and the pressure with the $L^{2}(0,t^{\ast};H^{1}(\Omega))$ norm.  Rates are calculated from the errors at two successive $\Delta t_{1,2}$ via
\begin{align*}
\frac{\log_{2}(e_{\chi}(\Delta t_{1})/e_{\chi}(\Delta t_{2}))}{\log_{2}(\Delta t_{1}/\Delta t_{2})},
\end{align*}
respectively, with $\chi = u, T, p$.  We set $m = 2\Delta t$ and vary $\Delta t$ between 8, 16, 24, 32, and 40.  Results are presented in Table 3.  Second order convergence is observed for velocity and temperature in the $L^{2}(0,t^{\ast};H^{1}(\Omega))$ norm and for pressure in the $L^{2}(0,t^{\ast};L^{2}(\Omega))$ norm, as predicted.  Moreover, third order convergence is seen for velocity and temperature $L^{\infty}(0,t^{\ast};L^{2}(\Omega))$ norm, whereby second order convergence is predicted in Theorem \ref{error:thick}.

\vspace{5mm}
\begin{adjustbox}{max width=\textwidth}
\begin{tabular}{ c  c  c  c  c  c  c  c  c  c  c  c }
	\hline			
	$1/m$ & $\vertiii{ <u_{h}>- u }_{\infty,0}$ & Rate & $\vertiii{ \nabla <u_{h}> - \nabla u }_{2,0}$ & Rate & $\vertiii{ <T_{h}> - T }_{\infty,0}$ & Rate & $\vertiii{ \nabla <T_{h}> - \nabla T }_{2,0}$ & Rate & $\vertiii{ <p_{h}> - p }_{2,0}$ & Rate \\
	\hline
	8 & 0.0005600 & - & 0.0206808 & - & 6.96E-05 & - & 0.0030380 & - & 0.0222107 & -\\
	16 & 6.28E-05 & 3.16 & 0.0046705 & 2.15 & 6.81E-06 & 3.35 & 0.0006157 & 2.30 & 0.0050407 & 2.14 \\
	24 & 1.82E-05 & 3.06 & 0.00209424 & 1.98 & 1.90E-06 & 3.14 & 0.0002505 & 2.22 & 0.0021921 & 2.05 \\
	32 & 6.99E-06 & 3.17 & 0.00102235 & 2.19 & 7.82E-07 & 3.12 & 0.0001264 & 2.23 & 0.0011483 & 2.13 \\
	40 & 3.99E-06 & 2.87 & 0.0007429 & 2.03 & 5.08E-07 & 2.59 & 9.25E-05 & 1.82 & 0.0007175 & 2.19 \\
	\hline  
\end{tabular}
\end{adjustbox}
\captionof{table}{Errors and rates for average velocity, temperature, and pressure in corresponding norms.} 
\section{Conclusion}
We presented an algorithm for calculating an ensemble of solutions to laminar natural convection problems.  This algorithm addresses both the competition between ensemble size and resolution in simulations and the need for higher order accurate timestepping methods.  In particular, the algorithm required the solution of two coupled linear systems, each involving a shared coefficient matrix, for multiple right-hand sides at each timestep.  Stability and convergence of the algorithm were proven and numerical experiments were performed to illustrate these properties.


\end{document}